\documentclass[11pt,a4paper,oneside]{amsart}

\usepackage[margin=3cm,top=3cm,bottom=3cm]{geometry}
\usepackage[british,UKenglish]{babel}
\usepackage[final]{pdfpages}

\usepackage{amsthm}
\usepackage{amsfonts}
\usepackage{amsmath}
\usepackage{amssymb}
\usepackage{mathtools}
\usepackage{centernot}
\usepackage{stmaryrd}

\usepackage{marginnote}
\usepackage{soul}
\usepackage[normalem]{ulem}
\mathtoolsset{showonlyrefs}
\usepackage{comment}
\usepackage{enumitem}
\usepackage{tikz-cd}

\usepackage{float}
\usepackage{import}
\usepackage{xifthen}
\usepackage{pdfpages}
\usepackage{transparent}
\usepackage{afterpage}
\usepackage{adjustbox}
\usepackage[pagewise]{lineno}

\usepackage{hyperref}
\hypersetup{colorlinks=true,citecolor=blue,urlcolor =black,linkbordercolor={1 0 0}}
\usepackage[style=alphabetic,maxbibnames=99,bibencoding=utf8]{biblatex}
\newtheorem{theorem}{Theorem}[section]

\newtheorem{corollary}[theorem]{Corollary}
\newtheorem{lemma}[theorem]{Lemma}
\newtheorem{proposition}[theorem]{Proposition}

\newtheorem{introthm}{Theorem}

\theoremstyle{definition}
\newtheorem{definition}[theorem]{Definition}
\newtheorem{example}[theorem]{Example}
\newtheorem{notation}[theorem]{Notation}

\theoremstyle{remark}
\newtheorem{remark}[theorem]{Remark}

\newcommand{\vir}{\mathrm{vir}}
\newcommand{\red}{\mathrm{red}}
\newcommand{\C}{\mathbb{C}}
\newcommand{\Z}{\mathbb{Z}}
\newcommand{\A}{\mathbb{A}}
\newcommand{\Q}{\mathbb{Q}}

\renewcommand{\O}{\mathcal{O}}
\newcommand{\ddR}{\mathrm d_{\mathrm{dR}}}
\newcommand{\Mf}{\mathfrak{M}}

\newcommand{\MMHM}{\mathrm{MMHM}}
\newcommand{\MHM}{\mathrm{MHM}}

\renewcommand{\phi}{\varphi}

\DeclareMathOperator{\HC}{\mathrm{HC}}
\DeclareMathOperator{\id}{\mathrm{id}}
\DeclareMathOperator{\Hom}{\mathrm{Hom}}

\DeclareMathOperator{\rk}{\mathrm{rk}}
\DeclareMathOperator{\tr}{\mathrm{tr}}
\DeclareMathOperator{\tot}{\mathrm{tot}}
\DeclareMathOperator{\Sym}{\mathrm{Sym}}
\DeclareMathOperator{\Crit}{\mathrm{Crit}}
\DeclareMathOperator{\Spec}{\mathrm{Spec}}
\DeclareMathOperator{\Chow}{\mathrm{Chow}}
\DeclareMathOperator{\End}{\mathrm{End}}
\DeclareMathOperator{\length}{\mathrm{length}}

\DeclareMathOperator{\QCoh}{\mathsf{QCoh}}

\DeclareMathOperator{\HH}{\mathrm{HH}}

\DeclareMathOperator{\ch}{\mathrm{ch}}

\DeclareMathOperator{\Perv}{\mathsf{Perv}}
\DeclareMathOperator{\rat}{\mathrm{rat}}
\DeclareMathOperator{\IC}{\mathrm{IC}}
\DeclareMathOperator{\Sch}{\mathsf{Sch}}
\DeclareMathOperator{\coker}{\mathrm{coker}}
\DeclareMathOperator{\gr}{\mathrm{gr}}

\title{$\chi$-Independence for BPS Invariants of Warped Local Curves}
\author{Ana Pavlakovi\'c}

\begin{document}

\maketitle

\begin{abstract}
  A well-known conjecture of Toda's asserts that the BPS invariants associated to a Calabi--Yau threefold are independent of the chosen Euler characteristic. We answer the question in the affirmative for a certain class of local surfaces, which we call warped local curves. Along the way we show that the moduli stack of sheaves on a warped local curve can be written as the global critical locus of a function. We also show that the moduli stack of sheaves on a curve fibration over $\A^1$ is the global critical locus of a function. 
\end{abstract}

\section{Introduction}

Curve and sheaf counting frameworks like Gromov--Witten or Donaldson--Thomas theory have played an important role in the study of enumerative geometry. Gopakumar–Vafa (BPS) invariants are conjectured to give an integral reformulation of the Gromov–Witten theory of a Calabi–Yau threefold. The relevant Gromov–Witten invariants are indexed by the genus and the curve class and do not involve the Euler characteristic of a sheaf. By contrast, sheaf-theoretic constructions of BPS invariants use moduli spaces of one-dimensional sheaves with fixed Euler characteristic. For such a construction to recover the Gopakumar–Vafa invariants predicted by Gromov–Witten theory, its output should be independent of the Euler characteristic. This expectation is formalized by Toda's $\chi$-independence conjecture \cite[Conjecture 2.15]{Toda2023}.

We begin with the general set-up. 
Let $X$ be a smooth quasi-projective Calabi--Yau threefold over $\C$. For $\beta$ a curve class and $m$ an integer, let 
\begin{equation}
\label{eq:159}
  \mathfrak{M}_X^{\mathrm{ss}}(\beta,m)
\end{equation}
be the moduli stack of pure one-dimensional compactly supported semistable sheaves $F$ on $X$, fundamental cycle $[F] = \beta$, and Euler characteristic $\chi(F) = m$. We denote its good moduli space by 
\begin{equation}
\label{eq:175}
  p\colon \mathfrak{M}_X^{\mathrm{ss}}(\beta,m) \to M_X^{\mathrm{ss}}(\beta,m).
\end{equation}
The stack $\mathfrak{M}_X^{\mathrm{ss}}(\beta,m)$ carries the oriented d-critical structure induced by the $(-1)$-shifted symplectic structure on the derived moduli stack of sheaves on $X$ \cite{PantevToenVaquieVezzosi2013}. Hence, one obtains a perverse sheaf of vanishing cycles \cite{Joyce2015}, and therefore a BPS sheaf on the good moduli space \cite{BenBassatBravBussiJoyce2015}, which we denote by 
\begin{equation}
\label{eq:149}
  \phi_{M_X^{\mathrm{ss}}(\beta,m)}.
\end{equation}
Moreover, there is a Hilbert--Chow morphism 
\begin{equation}
\label{eq:148}
  \HC\colon M^{\mathrm{ss}}_X(\beta,m) \to \Chow_{\beta}(X),
\end{equation}
mapping a sheaf to its one-dimensional Fitting support cycle.

Following Maulik--Toda \cite{MaulikToda2018}, the generalised Gopakumar--Vafa invariants $n_{g,\beta,m}$ \cite{GopakumarVafa1998a}, \cite{GopakumarVafa1998b} are defined by the identity 
\begin{equation}
\label{eq:95}
  \sum_{i\in\Z} \chi\left(\mathcal{H}^i\left(\HC_{*}\phi_{M_{X}^{\mathrm{ss}}(\beta,m)}\right)\right) y^i = \sum_{g\geq0} n_{g,\beta,m} \left(y^{1/2} + y^{-1/2}\right)^{2g}.
\end{equation}
Toda's $\chi$-independence conjecture \cite[Conjecture 2.15]{Toda2023} predicts that the invariants $n_{g,\beta,m}$ are independent of the Euler characteristic $m$. Thus, it suffices to prove that the Laurent polynomial on the left-hand side of \eqref{eq:95} is independent of $m$. In fact, we prove that the coefficients
\begin{equation}
\label{eq:147}
  \chi\left(\mathcal{H}^i\left(\HC_{*}\phi_{M_X^{\mathrm{ss}}(\beta,m)}\right)\right) 
\end{equation}
are independent of $m$ for the threefolds described below. 

Let $C$ be a smooth projective curve over $\C$ and let $L$ be a line bundle on $C$ of degree $> 2g(C)-2$. Set
  \begin{equation}
  \label{eq:83}
    \pi\colon Y \coloneq \tot_C(L)\to C.
  \end{equation}
For a class $c \in H^1(Y,\omega_Y)$, denote by 
\begin{equation}
\label{eq:160}
  q\colon X_c \to Y
\end{equation}
the total space of the corresponding $\omega_Y$-torsor. Since $q$ is Zariski locally isomorphic over $Y$ to the total space of $\omega_Y$, the variety $X_c$ is a smooth quasi-projective Calabi--Yau threefold. We refer to such threefolds as \emph{warped local curves}. The main result of the paper is the following.

\begin{introthm}[= Corollary \ref{cor:chi-independence}]\label{thm:B}
  Let $r\in\mathbb{Z}_{>0}$, $m, m'  \in \Z$.
  For $n \in \Z$, let 
  \begin{equation}
  \label{eq:28}
    \HC \colon M_{X_c}^{\mathrm{ss}}(r,n) \to \Chow_{r[C]}(X_c)
  \end{equation}
  be the Hilbert--Chow morphism for compactly supported, pure one-dimensional, semistable sheaves $F$ on $X_c$ satisfying 
  \begin{equation}
  \label{eq:176}
    (\pi \circ q)_{\#} [F] = r[C] \qquad \text{and} \qquad \chi(F) = n.
  \end{equation}
  Then for every $i \in \Z$,
  \begin{equation}
  \label{eq:208}
    \chi\left(\mathcal{H}^i\left(\HC_{*}\phi_{M_{X_c}^{\mathrm{ss}}(r,m)}\right)\right) = \chi\left(\mathcal{H}^i\left(\HC_{*}\phi_{M_{X_c}^{\mathrm{ss}}(r,m' )}\right)\right).
  \end{equation}
  In particular, the Maulik--Toda BPS invariants $n_{g,r[C],m}$ of $X_c$ are independent of $m$.
\end{introthm}

Theorem \ref{thm:B} should be viewed as an extension of the local curve case treated by Kinjo--Koseki \cite{KinjoKoseki2026}. They prove Toda's $\chi$-independence conjecture for local curves
\begin{equation}
\label{eq:177}
  X = \tot_C(N), \qquad  \text{where} \quad \rk(N) =2, \; \det N \simeq \omega_C
\end{equation}
as well as a cohomological $\chi$-independence theorem for BPS cohomology of Higgs moduli spaces.
Their proof uses a global critical chart for the local Calabi--Yau threefolds constructed by Kinjo--Masuda \cite{KinjoMasuda2024}. Our setting keeps the surface 
\begin{equation}
\label{eq:161}
  Y  = \tot_C(L),
\end{equation}
but replaces the split local Calabi--Yau geometry $X = \tot_C(N)$ by an arbitrary $\omega_Y$-torsor $X_c \to Y$. This introduces a new comparison problem: $Y$ has a Hitchin morphism after identifying compactly supported sheaves on $Y$ with $L$-valued Higgs bundles, whereas $X_c$ has no intrinsic Hitchin map. The support of a sheaf on $X_c$ instead defines a point in the Hilbert--Chow variety 
\begin{equation}
\label{eq:193}
  \Chow_{r[C]}(X_c).
\end{equation}
Thus, the proof of Theorem \ref{thm:B} requires two ingredients: a global critical chart for the torsor $X_c$, and a comparison between the Hitchin geometry of $Y$ and the Hilbert--Chow geometry of $X_c$.

\begin{introthm}[= Theorem \ref{thm:global critical chart}]\label{thm:A}
  Let $k$ be a field of characteristic $0$. Let $C$ be a smooth projective curve over $k$, let $L$ be a line bundle on $C$, and let 
  \begin{equation}
  \label{eq:202}
    \pi\colon Y:=\tot_C(L)\longrightarrow C
  \end{equation}
  be the total space of $L$. Fix a class $c\in H^1(Y,\omega_Y)$ and denote by $q\colon X_c\to Y$ the total space of the $\omega_Y$-torsor associated to it.
  There exists a function 
  \begin{equation}
  \label{eq:203}
     f_c\colon \mathfrak{M}_Y \to \A^1
  \end{equation}
  on the derived moduli stack $\mathfrak{M}_Y$ of compactly supported perfect complexes on $Y$ such that there exists a natural equivalence of $(-1)$-shifted symplectic derived Artin stacks 
  \begin{equation}
  \label{eq:204}
    \mathfrak{M}_{X_c}\simeq \Crit (f_c).
  \end{equation}
\end{introthm}

Theorem \ref{thm:A} extends the global critical chart result of Kinjo--Masuda \cite[Theorem 5.2]{KinjoMasuda2024}. In the local curve case treated in \cite{KinjoMasuda2024}, the relevant torsor classes come from $H^0(C,\omega_C \otimes L^{-2}) \subset H^1(Y,\omega_Y)$. We instead treat arbitrary classes in 
\begin{equation}
\label{eq:145}
  H^1(Y,\omega_Y) \simeq \bigoplus_{k\geq 1} H^1(C, \omega_C \otimes L^{-k})
\end{equation}
for $Y  = \tot_C(L)$. The point is that every such class admits a de Rham primitive, which allows the use of \cite[Corollary 6.19]{BozecCalaqueScherotzke2024}.
On the semistable locus identified with $L$-valued Higgs bundles, the function is explicit. Writing 
\begin{equation}
\label{eq:194}
  c = \sum_{k\geq 1} \iota_k(c_k), \qquad c_k \in H^1(C,\omega_C\otimes L^{-k}) \cong H^0(C,L^{\otimes k})^{\vee},
\end{equation}
one has
\begin{equation}
\label{eq:158}
  f_c([(E,\theta)]) = \sum_{k\geq 1 } \frac{1}{k} c_k\left( \tr(\theta^k)\right).
\end{equation}
This trace formula is crucial for two reasons. First, it shows that $f_c$ factors through the Hitchin base 
\begin{equation}
\label{eq:195}
  B_Y = \bigoplus_{i=1}^r H^0(C,L^{\otimes i}).
\end{equation}
Second, $f_c$ is additive under direct sums. The factorisation through $B_Y$ allows us to use Maulik--Shen's cohomological $\chi$-independence theorem for $L$-valued Higgs bundles \cite{MaulikShen2023}, while additivity allows the use of the BPS decomposition theorem in the form used by Kinjo--Koseki \cite{DavisonMeinhardt2020,KinjoKoseki2026}.

A similar method also gives a global critical chart theorem beyond total spaces of line bundles over curves. The additional input needed for the construction is again the existence of the de Rham primitives for classes in $H^1(\omega)$. We prove this for \emph{curve fibrations over $\A^1$}. 

\begin{introthm}[= Corollary \ref{corollary:global critical chart curve fibration}]\label{thm:C}
  In this statement, we assume $k = \C$. Let $S$ be a smooth complex surface and $\pi\colon S \to \A^1$ a proper surjective morphism with connected one-dimensional fibres. Fix a class $c \in H^1(S,\omega_S)$ and denote by $q\colon X_c \to S$ the total space of the $\omega_S$-torsor associated to it. There exists a function 
  \begin{equation}
  \label{eq:191}
    f_c\colon \mathfrak{M}_S \to \A^1
  \end{equation}
  on the derived moduli stack $\mathfrak{M}_S$ of compactly supported perfect complexes on $S$ such that there exists a natural equivalence of $(-1)$-shifted symplectic derived Artin stacks 
  \begin{equation}
  \label{eq:600}
    \mathfrak{M}_{X_c} \simeq \Crit(f_c).
  \end{equation}
\end{introthm}

Theorem \ref{thm:C} is not used in the proof of Theorem \ref{thm:B}. Its role is to show that the global critical chart mechanism is not special to surfaces of the form $\tot_C(L)$. In the curve fibration setting, the surjectivity of the de Rham differential, which yields the global critical chart, is proved by combining the Hodge--de Rham spectral sequence with the vanishing of $H^3_{\mathrm{dR}}(S)$. This follows from the topology of a proper curve fibration over the contractible base $\A^1(\C)$. \\

We now outline the proof of Theorem \ref{thm:B}. 
By Theorem \ref{thm:A}, the global critical chart identifies the derived moduli stack of sheaves on $X_c$ with the critical locus of the function $f_c$ on the moduli stack $\mathfrak{M}_Y$ on the surface $Y = \tot_C(L)$. The closed embedding of the critical locus is induced by pushforward along the affine torsor map 
\begin{equation}
\label{eq:196}
  q\colon X_c \to Y.
\end{equation}
Thus, the threefold vanishing-cycle sheaf can be compared to the vanishing cycles of the explicit function $f_c$ on $\mathfrak{M}_Y$.

Identifying compactly supported sheaves on $Y$ of class $r[C]$ with $L$-valued Higgs bundles $(E,\theta)$ on $C$, the function $f_c$ factors through the Hitchin morphism 
\begin{equation}
\label{eq:150}
  h_Y\colon \mathfrak{M}_Y^{\mathrm{ss}}(r,m) \to B_Y\coloneq \bigoplus_{i=1}^r H^0(C,L^{\otimes i}),
\end{equation}
so we may write 
\begin{equation}
\label{eq:197}
  f_c = \tilde{f}_c\circ h_Y
\end{equation}
on the semistable locus.  Maulik--Shen then implies that the direct image 
\begin{equation}
\label{eq:151}
  h_{Y,*} \IC_{M_Y^{\mathrm{ss}}(r,m)}
\end{equation}
is independent of $m$. Since vanishing cycles commute with proper pushforward, the same is true after applying $\phi_{\tilde{f}_c}$.

The BPS decomposition theorem then identifies the BPS sheaf obtained from the threefold critical locus with the first perverse cohomology of the vanishing-cycle direct image on the surface side. Therefore,
\begin{equation}
\label{eq:152}
  h_{X_c,*} \phi_{M_{X_c}^{\mathrm{ss}}(r,m)}, \qquad h_{X_c}\coloneq h_Y\circ \iota
\end{equation}
is independent of $m$.

It remains to pass from $h_{X_c}$ to the Hilbert--Chow morphism on $X_c$. 
We prove that every compact curve in $X_c$ is uniquely determined by its image in $Y$. Hence, pushforward of cycles along $q$ induces a finite morphism 
\begin{equation}
\label{eq:155}
  b\colon S_c \to B_Y, \qquad S_c\coloneq \operatorname{im}(\HC),
\end{equation}
and the diagram 
\begin{equation}
\label{eq:156}
\begin{tikzcd}
  M_{X_c}^{\mathrm{ss}}(r,m) \arrow[r,"\iota"] \arrow[d,"\HC" swap] & M_Y^{\mathrm{ss}}(r,m) \arrow[d,"h_Y"] \\ S_c \arrow[r,"b"] & B_Y
\end{tikzcd}
\end{equation}
commutes. Since $b$ is finite, Euler characteristics of the cohomology sheaves of $\HC_{*} \phi_{M_{X_c}^{\mathrm{ss}}(r,m)}$ can be computed after applying $b_{*}$. This reduces the final numerical statement to the already proven $\chi$-independence of 
\begin{equation}
\label{eq:157}
  h_{X_c,*} \phi_{M_{X_c}^{\mathrm{ss}}(r,m)}.
\end{equation}

\subsection*{Relation to Other Work}
Bousseau formulated a cohomological $\chi$-independence conjecture for one-dimensional sheaves on surfaces, and Maulik--Shen proved it for sheaves in ample classes on toric del Pezzo surfaces and for $L$-valued Higgs bundles when $\deg L>2g(C)-2$ \cite{Bousseau2022,MaulikShen2023}. Their theorem refines earlier degree-independence phenomena for Higgs moduli spaces, including the Hausel--Thaddeus mirror-symmetry picture and its refinements by Mellit and Groechenig--Wyss--Ziegler \cite{HauselThaddeus2003,Mellit2020,GroechenigWyssZiegler2020}.

The BPS decomposition formalism has been extended in the general setting of symmetric stacks by Bu--Davison--Ib\'a\~nez N\'u\~nez--Kinjo--P\u{a}durariu and Hennecart--Kinjo, giving a conceptual framework for the BPS summands and integrality decompositions appearing in cohomological Donaldson--Thomas theory \cite{BuDavisonIbanezNunezKinjoPadurariu2025,HennecartKinjo2025}. There are also complementary $p$-adic and categorical perspectives, including the $p$-adic BPS invariants of Carocci--Orecchia--Wyss and the quasi-BPS categories of P\u{a}durariu--Toda \cite{CarocciOrecchiaWyss2024,PadurariuToda2023,PadurariuToda2025,PadurariuToda2026}.
Moreover, Davison--Hennecart--Kinjo--Schiffmann--Vasserot have shown $\chi$-independence for quasi-projective symplectic surfaces \cite{DavisonHennecartKinjoSchiffmannVasserot2026}.

The derived geometry tools behind Theorems \ref{thm:A} and \ref{thm:C} comes from shifted symplectic geometry and Calabi--Yau categories. Pantev--To\"en--Vaqui\'e--Vezzosi construct shifted symplectic structures on derived moduli stacks, including the $(-1)$-shifted symplectic structures arising on moduli of objects in Calabi--Yau threefold categories \cite{PantevToenVaquieVezzosi2013}. Brav--Dyckerhoff formulate the corresponding statement for Calabi--Yau structures on dg categories and moduli of objects \cite{BravDyckerhoff2019}, \cite{BravDyckerhoff2021}. Keller introduces deformed Calabi--Yau completions of dg categories \cite{Keller2011}. Bozec--Calaque--Scherotzke identify moduli of objects in deformed Calabi--Yau completions with relative or twisted critical loci in a form suited to dimensional reduction \cite{BozecCalaqueScherotzke2024}. Kinjo--Masuda then give the geometric interpretation: deformed Calabi--Yau completions of categories of sheaves on a smooth variety correspond to categories of sheaves on total spaces of torsors of the canonical line bundle \cite{KinjoMasuda2024}.
Theorems \ref{thm:A} and \ref{thm:C} are obtained by combining these inputs with the existence of the de Rham primitives in the two geometric situations of the total space of a line bundle and curve fibrations over $\A^1$.

\subsection*{Plan of the Paper}

Section \ref{sec:preliminaries} recalls the Higgs-theoretic, d-critical, vanishing-cycle, mixed-Hodge-module, and Hochschild homology background used later. Section \ref{sec:global critical chart} contains the proof of the global critical chart theorems, namely Theorem \ref{thm:A} (= Theorem \ref{thm:global critical chart}) for warped local curves and Theorem \ref{thm:C} (= Corollary \ref{corollary:global critical chart curve fibration}) for curve fibrations over $\A^1$. Section \ref{sec:chi-independence} contains the Hitchin/Hilbert--Chow comparison and the proof of Theorem \ref{thm:B} (= Corollary \ref{cor:chi-independence}).

\subsection*{Notation and Conventions}

Unless specified otherwise, $k$ is a field of characteristic $0$. All statements pertaining to sheaves of vanishing cycles, mixed Hodge modules, or BPS sheaves are over the complex numbers. Throughout, $\mathcal{H}^i$ denotes perverse cohomology.
Subscripts and superscripts are omitted when the context prevents ambiguity. In particular, the symbols $B_Y, h_Y, h_{X_c}, \HC, S_c,$ and $b$ are always used after the class $c$, the rank $r$, and the Euler characteristic have been fixed. Thus, for example, when an Euler characteristic $m$ is replaced by $m' $, the same symbols denote the corresponding objects for the new Euler characteristic.

\subsection*{Acknowledgements}

The author would like to thank Tony Pantev for many helpful discussions and his invaluable guidance that helped shape this work. The author is partially supported by NSF FRG grant DMS-2244978 and NSF/BSF grant DMS-2200914.

\section{Preliminaries}\label{sec:preliminaries}

\subsection{Set-Up}

Let $k$ be a field of characteristic $0$ and let $C$ be a smooth projective curve over $k$. Let $L$ be a line bundle on $C$ and  set
\begin{equation}
\label{eq:45}
  \pi\colon  Y\coloneq \tot_C(L) = \Spec_C(\Sym L^{\vee})\to C.
\end{equation}
For a class $c \in H^1(Y,\omega_Y)$, let $q\colon X_c \to Y$ be the total space of the $\omega_Y$-torsor corresponding to $c$. The morphism $q$ is affine and $X_c$ is Zariski locally isomorphic to $\tot_Y(\omega_Y)$ over $Y$. If $c = 0$, then $X_c$ is isomorphic to $\tot_Y(\omega_Y)$ as an $\omega_Y$-torsor.

For any smooth variety $Z$, $\mathfrak{M}_Z\coloneq \mathfrak{M}_{\QCoh(Z)}$ denotes an appropriate version of the derived moduli stack of objects in the stable $\infty$-category $\QCoh(Z)$; it parametrises compactly supported perfect complexes on $Z$. If it exists, the good moduli space of $\mathfrak{M}_Z$ is denoted by $M_Z$ and the good moduli morphism by
\begin{equation}
\label{eq:162}
  p\colon \mathfrak{M}_Z \to M_Z.
\end{equation}

\begin{remark}\label{rem:factor-through-good-moduli}
  By the universal property of good moduli spaces, all morphisms from the stack to a scheme used below factor uniquely through the good moduli space, see \cite[Theorem 4.16.(iv)]{Alper2013}. By abuse of notation we frequently define maps on moduli spaces by providing the assignment for the moduli stack and omitting the factorisation.  
\end{remark}

\subsection{Hitchin, Hilbert--Chow, and Gopakumar--Vafa inputs}

\subsubsection{$L$-Valued Higgs Bundles}\label{sec:twisted higgs bundles}

For $r \in \Z_{>0}$ and $m \in \Z$, let 
\begin{equation}
\label{eq:163}
  \mathfrak{M}_Y^{\mathrm{ss}}(r,m) 
\end{equation}
be the derived moduli stack of compactly supported Gieseker semistable sheaves $G$ on $Y$ of pure dimension $1$, satisfying 
\begin{equation}
\label{eq:170}
  [G] = r[C] \in A^{1}(Y) 
\end{equation}
and  $\chi(G) = m$. Equivalently, $\mathfrak{M}_Y^{\mathrm{ss}}(r,m)$ is the moduli stack of slope-semistable $L$-valued Higgs sheaves $(E,\theta)$ on $C$, where $E$ is locally free 
\begin{equation}
\label{eq:115}
  \rk(E) = r, \qquad \chi(E) = m, \qquad   \theta\in H^0(C, \End E \otimes L).
\end{equation}
By the spectral correspondence, compactly supported pure one-dimensional sheaves on $Y$ of class $r[C]$ are equivalent to rank $r$ $L$-valued Higgs bundles on $C$, compatible with Gieseker semistability.

Note that Gieseker and slope-stability are the same for Higgs bundles on $C$ but they are not the same for the pure sheaves on $Y$. So we can also view the spectral correspondence as a dictionary translating information about Gieseker semistable pure compactly supported sheaves of class $r[C]$ and Euler characteristic $m$ to information about slope-semistable $L$-valued Higgs bundles of rank $r$ and Euler characteristic $m$.

The \emph{Hitchin base} is the vector space 
\begin{equation}
\label{eq:43}
  B_Y\coloneq \bigoplus_{k=1}^r H^0(C,L^{\otimes k}).
\end{equation}
Denote by $M_Y^{\mathrm{ss}}(r,m)$ the good moduli space of pure sheaves
\begin{equation}
\label{eq:44}
  p\colon \mathfrak{M}_Y^{\mathrm{ss}}(r,m) \to M_Y^{\mathrm{ss}}(r,m).
\end{equation}
For a geometric point $(E,\theta)$ of $\mathfrak{M}^{\mathrm{ss}}_Y(r,m)$, define 
\begin{equation}
\label{eq:146}
    \tilde{h}_Y(E,\theta) \coloneq (\tr(\theta), \tr(\theta^2),\ldots,\tr(\theta^r)) \in B_Y,
\end{equation}
where $\tr(\theta^i) \in H^0(C, L^{\otimes i})$. This gives a morphism $\tilde{h}_Y\colon \mathfrak{M}^{\mathrm{ss}}_Y(r,m) \to B_Y$. By Remark \ref{rem:factor-through-good-moduli}, $\tilde{h}_Y$ factors uniquely as
\begin{equation}
\label{eq:42}
  \mathfrak{M}_Y^{\mathrm{ss}}(r,m) \xrightarrow{p} M_Y^{\mathrm{ss}}(r,m) \xrightarrow{h_Y} B_Y.
\end{equation}
The morphism $h_Y$ is the \emph{Hitchin morphism}.
There is an isomorphism 
\begin{equation}
\label{eq:48}
  \rho_Y\colon B_Y \to   \Chow_{r[C]}(Y),\qquad a \mapsto [C_a]
\end{equation}
where $C_a \subset Y$ is the spectral curve determined by $a$ \cite[Section 3]{BeauvilleNarasimhanRamanan1989}.

\subsubsection{Warped Local Curves}\label{sec:warped local curves}
Fix $r \in \Z_{>0}$ and $m \in \Z$. Let
\begin{equation}
\label{eq:40}
  \mathfrak{M}_{X_c}^{\mathrm{ss}}(r,m)
\end{equation}
be the Artin stack of compactly supported, pure one-dimensional, slope-semistable sheaves $F$ on $X_c$ satisfying 
\begin{equation}
\label{eq:165}
  (\pi\circ q)_{\#}[F] = r[C], \qquad \chi(F) = m.
\end{equation} 
Its good moduli space is denoted by
\begin{equation}
\label{eq:41}
  p\colon \mathfrak{M}_{X_c}^{\mathrm{ss}}(r,m) \to M_{X_c}^{\mathrm{ss}}(r,m).
\end{equation}
The analogue of the Hitchin morphism on $X_c$ is the \emph{Hilbert--Chow morphism}. We write $\Chow_{r[C]}(X_c)$ for the \emph{Chow scheme} of compact one-cycles $\Gamma$ on $X_c$ satisfying 
\begin{equation}
\label{eq:166}
  (\pi\circ q)_{\#}[\Gamma] = r[C].
\end{equation}
Corollary \ref{corollary:existence curve class} shows that $\Chow_{r[C]}(X_c)$ is non-empty assuming a mild condition on $c$. 
Then 
\begin{equation}
\label{eq:207}
  \HC\colon M_{X_c}^{\mathrm{ss}}(r,m) \to \Chow_{r[C]}(X_c)
\end{equation}
sends a sheaf to its fundamental support cycle. 

Explicitly, for a sheaf $F$, write the support as a union of irreducible one-dimensional components
\begin{equation}
\label{eq:111}
  \mathrm{supp}(F) = \bigcup_i Z_i.
\end{equation}
If $\eta_i$ is a generic point of $Z_i$, define 
\begin{equation}
\label{eq:200}
  n_i \coloneq \length_{\O_{X_c,\eta_i}}(F_{\eta_i}).
\end{equation}
Then
\begin{equation}
\label{eq:68}
  \HC(F) \coloneq \sum_i n_i[Z_i] \in  \Chow_{r[C]}(X_c).
\end{equation}

There is a pushforward morphism of Chow spaces
\begin{equation}
\label{eq:53}
  q_{\#}\colon \Chow_{r[C]}(X_c) \to \Chow_{r[C]}(Y)
\end{equation}
given on cycles by 
\begin{equation}
\label{eq:167}
  \sum_i n_i[Z_i] \mapsto \sum_i n_i \cdot \deg (Z_i \to q(Z_i)) \cdot [q(Z_i)].
\end{equation}

Let 
\begin{equation}
\label{eq:70}
  S_c\hookrightarrow \Chow_{r[C]}(X_c)
\end{equation}
be the scheme-theoretic image of the Hilbert--Chow morphism for the fixed triple $(c,r,m)$. We denote the induced morphism by 
\begin{equation}
\label{eq:215}
  \HC\colon M_{X_c}^{ss}(r,m)\longrightarrow S_c.
\end{equation}
We show in Section \ref{sec:semistability} that pushforward of sheaves along the affine morphism $q$ preserves the relevant semistable loci. Thus, it induces a morphism on good moduli spaces
\begin{equation}
\label{eq:71}
  \iota\colon M_{X_c}^{\mathrm{ss}}(r,m) \to M_Y^{\mathrm{ss}}(r,m).
\end{equation}
Define
\begin{equation}
\label{eq:72}
  h_{X_c} \coloneq h_Y\circ \iota \colon M_{X_c}^{\mathrm{ss}}(r,m) \to B_Y.
\end{equation}
For every sheaf $F$ in $\mathfrak{M}_{X_c}^{\mathrm{ss}}(r,m)$, one has 
\begin{equation}
\label{eq:168}
  q_{\#} \HC(F) =\rho_Y(h_{X_c}(F)).
\end{equation}
Hence,  $q_{\#}S_c \subset \operatorname{im}(\rho_Y)$.
Composing $q_{\#}$ with the inverse of $\rho_Y$, we obtain  
\begin{equation}
\label{eq:50}
  b\coloneq \rho_Y^{-1}  \circ q_{\#}|_{S_c}\colon S_c \to B_Y.
\end{equation}
The resulting diagram commutes:
\begin{equation}
\label{eq:51}
\begin{tikzcd}
  M_{X_c}^{\mathrm{ss}} (r,m) \arrow[r,"\iota"] \arrow[d,"\HC" swap] \arrow[dr, "h_{X_c}"]& M_Y^{\mathrm{ss}}(r,m) \arrow[d,"h_Y"] \\ S_c \arrow[r,"b"] & B_Y. 
\end{tikzcd}
\end{equation}

\subsection{D-Critical Structures and Vanishing Cycles}
Let $k = \C$.
The BPS sheaf $\phi_M$ is built from the vanishing cycles of an oriented d-critical stack. We recall the necessary background information on d-critical charts and vanishing cycles to explain the construction.

\subsubsection{D-Critical Charts}

Let $X$ be a $\C$-scheme locally of finite type. Joyce constructs in  \cite[Theorem 2.1]{Joyce2015}  a sheaf $\mathcal{S}_X$ with the following local description. For a Zariski open subscheme $R \subset X$ and a closed embedding $i \colon R \hookrightarrow U$  into a smooth $\C$-scheme, let $I_{R,U} \subset i^{-1}\O_U$ be the ideal sheaf of $R$ in $U$. There is an exact sequence
\begin{equation}
\label{eq:92}
  0 \to \mathcal{S}_X|_R \xrightarrow{\iota_{R,U}} i^{-1} \O_U / I^2_{R,U} \xrightarrow{\ddR} i^{-1}\Omega^1_U/ (I_{R,U} \cdot i^{-1} \Omega^1_U),
\end{equation}
where the last morphism is induced by the de Rham differential. 
The natural morphism $\mathcal{S}_X \to \O_X$, obtained by locally composing $\iota_{R,U}$ with $i^{-1}\O_U/I^2_{R,U} \twoheadrightarrow \O_R$, defines a subsheaf
\begin{equation}
\label{eq:93}
  \mathcal{S}^0_X \coloneq \ker \left( \mathcal{S}_X \to \O_X \twoheadrightarrow \O_{X_{\red}}\right) \subset \mathcal{S}_X.
\end{equation}

\begin{definition}
  A \emph{d-critical structure} on $X$ is a section $s \in H^0(X,\mathcal{S}_X^0)$ such that every point of $X$ has a Zariski open neighbourhood  $R \subset X$ for which there exist a smooth $\C$-scheme $U$, a closed embedding $i\colon R \hookrightarrow U$, and a regular function $f\colon U \to \A^1$ satisfying 
  \begin{equation}
  \label{eq:8}
    i(R) = \Crit(f) \subset U
  \end{equation}
  as closed subschemes and
  \begin{equation}
    \label{eq:94}
    \iota_{R,U}(s|_R) = i^{-1}(f) + I^2_{R,U} \qquad \text{in} \quad i^{-1}\O_U / I^2_{R,U}.
  \end{equation}
  In particular, $f|_{R_{\red}} = 0$. The quadruple $(R,U,f,i)$ is called a \emph{d-critical chart} on $(X,s)$ and $(X,s)$ is called a \emph{d-critical scheme}.
\end{definition}

\begin{remark}
  If $q\colon X_1 \to X_2$ is a smooth morphism of schemes, Joyce's functoriality provides a natural morphism
  \begin{equation}
  \label{eq:39}
    q^{*}\colon q^{-1} \mathcal{S}_{X_2}^0 \to \mathcal{S}_{X_1}^0.
  \end{equation}
  For $s_2 \in H^0(X_2, \mathcal{S}^0_{X_2})$, write $q^{*}s_2 \coloneq q^{*}(s_2)$. If $(X_2,s_2)$ is a d-critical scheme, then $(X_1,q^{*}s_2)$ is a d-critical scheme. Conversely, if $q$ is also surjective and $(X_1,q^{*}s_2)$ is d-critical, then so is $(X_2,s_2)$. These pullbacks are compatible with composition.
\end{remark}

For an Artin stack $\mathfrak{X}$ locally of finite type over $\C$, Joyce defines a sheaf $\mathcal{S}^0_{\mathfrak{X}}$ on the lisse-\'etale site of $\mathfrak{X}$. On an object $t\colon T \to \mathfrak{X}$, where $T$ is a scheme and $t$ is smooth, it is given by $\mathcal{S}^0_T$, with transition morphisms induced by smooth pullback. Thus, a global section $s \in H^0(\mathfrak{X}, \mathcal{S}^0_{\mathfrak{X}})$ determines compatible sections 
\begin{equation}
\label{eq:14}
  s_T\coloneq t^{*}s \in H^0(T, \mathcal{S}^0_T)
\end{equation}
for all smooth morphisms $t\colon T \to \mathfrak{X}$.

\begin{definition}
  A \emph{d-critical Artin stack} is a pair $(\mathfrak{X},s)$ consisting of an Artin stack $\mathfrak{X}$ locally of finite type over $\C$ and a section $s \in H^0(\mathfrak{X}, \mathcal{S}^0_{\mathfrak{X}})$ such that $(T,s_T)$ is a d-critical scheme for every smooth morphism $t\colon T \to \mathfrak{X}$ from a scheme $T$. 
\end{definition}

\begin{remark}
  The d-critical structures used in this paper arise from classical truncations of $(-1)$-shifted symplectic derived Artin stacks. More precisely, if $\mathbf{\mathfrak{Z}}$ is a $(-1)$-shifted symplectic derived Artin stack and $\mathfrak{X} = t_0(\mathbf{\mathfrak{Z}})$, then $\mathfrak{X}$ carries a natural d-critical structure $s$ \cite[Theorem 3.18]{BenBassatBravBussiJoyce2015}. Furthermore, the line bundle $\det (\mathbb{L}_{\mathfrak{Z}})|_{ \mathfrak{X}_{\red}}$ depends only on the pair $( \mathfrak{X},s)$ and not on the full derived stack $\mathbf{\mathfrak{Z}}$. We denote this line bundle by $K^{\vir}_{\mathfrak{X},s}$. In fact, Joyce shows  that the line bundle $K^{\vir}_{\mathfrak{X},s}$ exists without assuming that $(\mathfrak{X},s)$ arises as a truncation.
  The precise statement can be summarised as follows.
\end{remark}

\begin{theorem}[{\cite[Theorems 2.28 and 2.56]{Joyce2015}}]
  Let $(X,s)$ be a d-critical scheme. There is a line bundle $K^{\vir}_{X,s}$ on $X_{\red}$, unique up to canonical isomorphism, such that every d-critical chart $(R,U,f,i)$ carries a natural isomorphism 
  \begin{equation}
  \label{eq:55}
  \kappa_{R,U}\colon K^{\vir}_{X,s}|_{R_{\red}} \xrightarrow{\sim} (i|_{R_{\red}})^{*} K^{\otimes 2}_U,
  \end{equation}
  where $K_U$ is the canonical bundle of the smooth scheme $U$.

  Let $(\mathfrak{X},s)$ be a d-critical Artin stack. There is likewise a line bundle $K^{\vir}_{\mathfrak{X},s}$ on the reduced substack $\mathfrak{X}_{\red}$, unique up to canonical isomorphism. It is characterised by the following property: For every smooth morphism $t\colon T \to \mathfrak{X}$ from a scheme, there is a natural isomorphism of line bundles on $T_{\red}$, 
  \begin{equation}
  \label{eq:67}
    \Gamma_t\colon t^{*}_{\red} K^{\vir}_{\mathfrak{X},s} \otimes \left( \det( \Omega^1_{T/ \mathfrak{X}})|_{T_{\red}}\right)^{\otimes 2} \xrightarrow{\sim} K^{\vir}_{T,s_T}.
  \end{equation}
  These isomorphisms are compatible with smooth morphisms between schemes over $\mathfrak{X}$ and with composition.
\end{theorem}

\begin{definition}
  An \emph{orientation} of a d-critical scheme or a d-critical Artin stack $(\mathfrak{X},s)$ is a line bundle $L$ on $\mathfrak{X}_{\red}$ together with an isomorphism 
  \begin{equation}
  \label{eq:97}
    o\colon L^{\otimes 2} \xrightarrow{\sim} K^{\vir}_{\mathfrak{X},s}.
  \end{equation}
  The triple $(\mathfrak{X},s,o)$ is then called an \emph{oriented d-critical stack}. If $t\colon T \to \mathfrak{X}$ is smooth, the induced orientation on the d-critical scheme $(T,s_T)$ has orientation line 
  \begin{equation}
  \label{eq:164}
    L_T\coloneq t^{*}_{\red}L \otimes \det (\Omega^1_{T/ \mathfrak{X}})|_{T_{\red}}.
  \end{equation}
  Its square is identified with $K^{\vir}_{T,s_T}$ by the composition 
  \begin{equation}
  \label{eq:190}
    L^{\otimes 2}_T \xrightarrow{t^{*}_{\red}\otimes \id} t^{*}_{\red} K^{\vir}_{\mathfrak{X},s} \otimes \left( \det (\Omega^1_{T/ \mathfrak{X}})|_{T_{\red}}\right)^{\otimes 2} \xrightarrow{\Gamma_t} K^{\vir}_{T,s_T}.
  \end{equation}
  For a smooth morphism $q\colon T_1 \to T_2$ over $\mathfrak{X}$, these induced orientation lines satisfy the compatibility 
  \begin{equation}
  \label{eq:206}
    L_{T_1} \cong q^{*}_{\red} L_{T_2}\otimes \det (\Omega^1_{T_1/T_2}) |_{T_{1,\red}}.
  \end{equation}
\end{definition}

\subsubsection{Vanishing Cycles}
Throughout this section, schemes and Artin stacks are defined over $\C$. 

\begin{definition}
  Let $U$ be a smooth complex algebraic variety and let $f\colon U \to \A^1$ be a regular. For $a \in \A^1$, write $U_a \coloneq f^{-1}(a)$. We denote by 
  \begin{equation}
  \label{eq:62}
    \phi_{f-a} \colon D^{b}_c(U,\Q) \to D^b_c(U_a,\Q)
  \end{equation}
  the perverse-normalised vanishing-cycle functor. Thus, $\phi_{f-a}$ sends perverse sheaves to perverse sheaves. In particular, 
  \begin{equation}
  \label{eq:102}
    \phi_{f-a} \left( \Q_U[\dim U]\right)
  \end{equation}
  is a perverse sheaf on $U_a$ supported on $\Crit(f) \cap U_a$.
  In the finite type algebraic situations considered below, the set of critical values of $f$ is finite. For $P \in D^b_c(U,\Q)$, we use the notation 
  \begin{equation}
  \label{eq:118}
    \phi_f(P)\coloneq \bigoplus_{a \in f(\Crit(f))} \phi_{f-a}(P),
  \end{equation}
  where each summand is regarded as a complex on $\Crit(f)$ and is extended by zero away from $\Crit(f) \cap U_a$. To ease notation, we use the shorthand 
  \begin{equation}
  \label{eq:119}
    \phi_f\coloneq \phi_f \left( \Q_U[\dim U]\right) \in \Perv(\Crit(f)).
  \end{equation}
  When $0$ is the only critical values, this agrees with the usual notation for the vanishing cycles of $f$.
\end{definition}

\begin{definition}
  Let $\mathcal{U}$ be a smooth complex Artin stack locally of finite type. An object $P \in \Perv( \mathcal{U})$ is given by perverse sheaves $P_T \in \Perv(T)$ for all smooth morphisms $t\colon T \to \mathcal{U}$ from schemes, together with compatible isomorphisms 
  \begin{equation}
  \label{eq:120}
    q^{*} P_{T_2}[\dim q] \xrightarrow{\sim} P_{T_1}
  \end{equation}
  for every smooth morphism $q\colon (t_1\colon T_1 \to \mathcal{U}) \to (t_2\colon T_2 \to \mathcal{U})$.

  Let $f\colon \mathcal{U} \to \A^1$ be a regular function. For $P \in \Perv(\mathcal{U})$, the objects 
  \begin{equation}
  \label{eq:126}
    \phi_{f\circ t} (P_T)
  \end{equation}
  are identified. They therefore define a perverse sheaf $\phi_f(P) \in \Perv(\Crit(f))$. For a smooth stack, we write
  \begin{equation}
  \label{eq:229}
    \phi_f\coloneq \phi_f \left(\Q_{\mathcal{U}} [\dim \mathcal{U}]\right).
  \end{equation}
\end{definition}

\begin{theorem}[{\cite[Theorem 6.9]{BravBussiDupontJoyceSzendroi2015}}]\label{thm:dc-space-vc}
  Let $(X,s,o)$ be an oriented d-critical space with orientation $L$.
  There is a perverse sheaf $\phi_{X,s,o} \in \Perv(X)$, natural up to canonical isomorphism, with the following local description. 
  Let $(R,U,f,i)$ be a d-critical chart on $(X,s)$. The chart gives an isomorphism 
  \begin{equation}
  \label{eq:228}
    \iota_{R,U,f,i}\colon K^{\vir}_{X,s}|_{R_{\red}}\xrightarrow{\sim} i^*K_U^{\otimes2}|_{R_{\red}}.
  \end{equation}
  Let $Q_{R,U,f,i} \to T$ be the principal $\Z_2$-bundle whose local sections  are isomorphisms 
  \begin{equation}
  \label{eq:52}
    \alpha\colon L|_{R_{\red}}\xrightarrow{ \sim }i^*K_U|_{R_{\red}}
  \end{equation}
  satisfying 
  \begin{equation}
  \label{eq:56}
    \alpha^{\otimes 2}=\iota_{R,U,f,i}\circ o|_{R_{\red}}.
  \end{equation}
  Then there is a natural local identification 
  \begin{equation}
  \label{eq:59}
    \phi_{X,s,o} |_R \xrightarrow{\sim} i^{*}\phi_f\otimes_{\Z_2} Q_{R,U,f,i}.
  \end{equation}
  Here $\otimes_{\Z_2} Q_{R,U,f,i}$ denotes twisting by the rank $1$ local system associated to the sign representation of $\Z_2$. These local identifications are compatible under embeddings of d-critical charts and therefore glue to the global perverse sheaf $\phi_{X,s,o}$.
\end{theorem}

\begin{example}
  Suppose that $X=\Crit(f)$ globally for a regular function $f\colon U\to\A^1$ on a smooth variety $U$, and equip $X$ with the induced d-critical structure. The chart identifies 
  \begin{equation}
  \label{eq:227}
    K^{\vir}_{X,s}\cong K_U^{\otimes 2}|_{X_{\red}}.
  \end{equation}
  For the corresponding canonical orientation, whose orientation line is $K_U|_{X_{\red}}$, the principal $\Z_2$-bundle in Theorem~\ref{thm:dc-space-vc} has a canonical trivialization. Hence,
\begin{equation}
\label{eq:124}
  \phi_{X,s,o} \simeq \phi_f|_X.
\end{equation}
\end{example}

\begin{remark}\label{rem:smooth-pullback-space}
  This construction is compatible with smooth pullback. If $q\colon(X_1,s_1,o_1)\longrightarrow (X_2,s_2,o_2)$ is a smooth morphism of oriented d-critical spaces of relative dimension $d$, where $s_1 = q^{*}s_2$ and $o_1 = q^{*}o_2$, then there is a canonical isomorphism 
\begin{equation}
\label{eq:125}
  \Theta_q\colon \phi_{X_1,s_1,o_1} \xrightarrow{\sim} q^{*}\phi_{X_2,s_2,o_2}[d],
\end{equation}
compatible with composition and with the smooth base-change isomorphisms on d-critical charts.
\end{remark}

\begin{theorem}[{\cite[Theorem 4.8]{BenBassatBravBussiJoyce2015}}]\label{thm:global sheaf of vanishing cycles}
  Let $(\mathcal{X},s,o)$ be an oriented d-critical Artin stack locally of finite type over $\C$. There exists a natural perverse sheaf $\phi_{\mathcal{X},s,o} \in \Perv( \mathcal{X})$ characterised by the following property: For every smooth morphism $t\colon T \to \mathcal{X}$ from a scheme of relative dimension $d$, let $t^{*}s$ be the induced d-critical structure on $T$ and let $t^{*}o$ be the pull back of the orientation. Then there is a natural isomorphism 
  \begin{equation}
  \label{eq:226}
    \theta_t \colon \phi_{T,t^{*}s,t^{*}o} \xrightarrow{ \sim} t^{*}\phi_{\mathcal{X},s,o}[d].
  \end{equation}
  For every smooth morphism $q\colon (t_1\colon T_1 \to \mathcal{X}) \to (t_2\colon T_2 \to \mathcal{X})$ these isomorphisms are compatible with the pullback $\theta_q$ in Remark \ref{rem:smooth-pullback-space}.

  Thus, the family of perverse sheaves on the smooth charts of $\mathcal{X}$ satisfies the cocycle condition and determines $\phi_{\mathcal{X},s,o}$ uniquely up to canonical isomorphism. 
\end{theorem}

\begin{example}
  Let $U$ be a smooth Artin stack locally of finite type and let $f\colon U \to \A^1$ be a regular function. The derived critical locus $\Crit(f)$ carries a canonical $(-1)$-shifted symplectic structure. Its classical truncation 
  \begin{equation}
  \label{eq:211}
    Z \coloneq t_0\Crit(f)
  \end{equation}
  has the induced d-critical structure and a canonical orientation whose orientation line is 
  \begin{equation}
  \label{eq:216}
    K_U|_{Z_{\red}}, \qquad K_U \coloneq \det \mathbb{L}_U
  \end{equation}
  by \cite[Lemma A.10]{KinjoKoseki2026}. Equivalently, this orientation is induced by the canonical isomorphism 
  \begin{equation}
  \label{eq:225}
    K_U^{\otimes 2}|_{Z_{\red}} \xrightarrow{\sim} K^{\vir}_{Z,s}.
  \end{equation}
  For this orientation, \cite[Section A.2]{KinjoKoseki2026} shows that there is a natural isomorphism
  \begin{equation}
    \label{eq:127}
    \phi_{\Crit(f),s,o}\xrightarrow{\sim} \phi_f(\Q_U[\dim U]).
  \end{equation}
  Thus, a global critical locus presentation produces the same vanishing-cycle sheaf as the intrinsic construction from the induced oriented d-critical stack. In particular, the resulting sheaf depends only on the induced oriented d-critical structure, not on the choice of a compatible critical locus presentation; see \cite[Theorem 6.9]{BravBussiDupontJoyceSzendroi2015}, \cite[Theorem 4.8]{BenBassatBravBussiJoyce2015}, and \cite[Lemma A.10 and Appendix A.2]{KinjoKoseki2026}.
\end{example}

\begin{notation}
  When $(s,o)$ are the natural d-critical structure and orientation on $\mathfrak{M}_{X_c}^{\mathrm{ss}}(r,m)$ constructed in Proposition \ref{prop:naturally oriented}, we write 
  \begin{equation}
  \label{eq:96}
    \phi_{ \mathfrak{M}_{X_c}^{\mathrm{ss}}(r,m)} \coloneq \phi_{\mathfrak{M}^{\mathrm{ss}}_{X_c}(r,m),s,o}.
  \end{equation}
\end{notation}

\subsection{Monodromic Mixed Hodge Modules}
Let $k  = \C$.

\subsubsection{Mixed Hodge Modules}

Let $X$ be a separated complex scheme locally of finite type with quasi-compact connected components. Saito's category of \emph{mixed Hodge modules} is denoted $\MHM(X)$ with bounded derived category $D^b(\MHM(X))$. There is an exact rationalisation functor 
\begin{equation}
\label{eq:98}
\rat\colon D^b(\MHM(X)) \to D^b( \Perv(X)),
\end{equation}
which restricts to a faithful functor $\MHM(X) \to \Perv(X)$ and is compatible with all six functors. 
Denote by $W_{\bullet}M$ the increasing \emph{weight filtration} of a mixed Hodge module $M$. It is \emph{pure of weight $i$} if $W_{i-1}M = 0$ and $W_i M = M$. A complex $M \in D^b(\MHM(X))$ is \emph{pure} if $\mathcal{H}^i(M)$ is pure of weight $i$ for every $i$. The category of pure Hodge modules is semisimple. Let $a_X\colon X \to \Spec \C$ be the structure morphism and let $\Q$ be the pure Hodge structure of weight zero and dimension one. Define
\begin{equation}
\label{eq:132}
  \Q_X \coloneq a_X^{*}\Q, \qquad D\Q_X\coloneq a_X^{!}\Q
\end{equation}
in $D^b(\MHM(X))$.

\begin{definition}
  For a complex Artin stack $\mathfrak{X}$, let $\Sch_{\mathfrak{X}}^{\mathrm{sm}, \mathrm{sep}}$ be the category of smooth separated  schemes of finite type over $\mathfrak{X}$. A \emph{mixed Hodge module} on $\mathfrak{X}$ is an assignment 
  \begin{equation}
  \label{eq:133}
    (t\colon T \to \mathfrak{X} ) \mapsto M_t \in \MHM(T)
  \end{equation}
  together with isomorphisms 
  \begin{equation}
  \label{eq:134}
    q^{*}M_{t_2}[\dim q] \simeq M_{t_1}
  \end{equation}
  for each smooth morphism $q\colon (t_1\colon T_1 \to \mathfrak{X}) \to (t_2\colon T_2 \to \mathfrak{X})$ in $\Sch^{\mathrm{sm},\mathrm{sep}}_{\mathfrak{X}}$, satisfying the usual associativity condition.  There is a rationalisation functor 
  \begin{equation}
  \label{eq:135}
    \rat\colon \MHM( \mathfrak{X}) \to \Perv( \mathfrak{X}).
  \end{equation}
\end{definition}

We use the approximation construction of \cite[Section 4.1]{KinjoKoseki2026} for pushforward to a good moduli space.

In fact, given a morphism of finite type $p\colon \mathfrak{X} \to X$ from an Artin stack to a separated scheme of finite type satisfying an approximation condition\footnote{For every constructible complex $F \in D^b_c(\mathfrak{X})$ and integer $N$, there exists a smooth morphism $q_N\colon T_N \to \mathfrak{X}$ from a scheme such that the natural maps on perverse cohomology are isomorphisms in degrees $\leq N$. },
one can define
\begin{equation}
\label{eq:136}
  \mathcal{H}^n(p_{*}M) \coloneq   \mathcal{H}^n((p\circ q_N)_{*} q^{*}_N M[\dim q_N])
\end{equation}
for $M \in \MHM(\mathfrak{X})$, where $q_N\colon T_N \to \mathfrak{X}$ is the choice of approximation, by choosing such a sufficiently high approximation and taking perverse cohomology on $X$. This is independent of the choice of approximation. The semistable moduli stacks used below are quotient stacks, so the approximation condition holds; see \cite[Section 4.1]{KinjoKoseki2026}.
For a complex $M \in D(\MHM(\mathfrak{X}))$ we also use 
\begin{equation}
\label{eq:137}
  \mathcal{H}(M)\coloneq \bigoplus_i \mathcal{H}^i(M)[-i].
\end{equation}
When $\mathcal{H}(p_{*}M)$ is pure and $h\colon X \to B$ is proper, Saito's decomposition theorem and  \cite[Corollary 5.2.14]{Dimca2004} give the compatibility 
\begin{equation}
\label{eq:138}
  \mathcal{H}(h_{*} \mathcal{H}(p_{*}M)) \simeq \mathcal{H}((h\circ p)_{*}M),
\end{equation}
see \cite[Section 4.1]{KinjoKoseki2026}.

\subsubsection{Monodromic Mixed Hodge Modules}
For a separated scheme $X$ locally of finite type over $\C$ with quasi-compact connected components, let $\MMHM(X)$ be the abelian category of \emph{monodromic mixed Hodge modules}. Informally, an object is a mixed Hodge module together with monodromy data. There is a forgetful functor 
\begin{equation}
\label{eq:139}
  \MMHM(X) \to \MHM(X)
\end{equation}
and a fully faithful embedding 
\begin{equation}
\label{eq:99}
  \MHM(X) \hookrightarrow \MMHM(X),
\end{equation}
which equips a mixed Hodge module with trivial monodromy. 

For a regular function $f\colon X \to \A^1$, there is a monodromic vanishing-cycle functor 
\begin{equation}
\label{eq:104}
  \phi_f^{\mathrm{mmhm}}\colon D^b(\MHM(X)) \to D^b(\MMHM(X))
\end{equation}
whose rationalisation is the constructible vanishing-cycle functor. The construction extends by smooth descent to Artin stacks and to oriented d-critical stacks. Thus, $\phi_{\mathfrak{X},s,o}$ has a monodromic and ordinary mixed Hodge module lifts 
\begin{equation}
\label{eq:105}
  \phi_{\mathfrak{X},s,o}^{\mathrm{mmhm}} \in \MMHM(\mathfrak{X}), \qquad \phi_{\mathfrak{X},s,o}^{\mathrm{mhm}} \in \MHM(\mathfrak{X});
\end{equation}
the latter is obtained by forgetting monodromy. 
We keep the monodromic mixed Hodge module whenever invoking the BPS decomposition theorem; applying $\rat$ gives the perverse sheaf used in the final numerical Euler characteristic statement.

Two structure features of the monodromic category are used later. First, the Thom--Sebastiani isomorphism holds for monodromic vanishing cycles \cite[Proposition 2.13]{DavisonMeinhardt2020}. For functions compatible with a direct-sum morphism, vanishing cycles of the sum are identified with the external product of the vanishing cycles of the summands, pushed forward along that direct-sum morphism. Second, the category contains a square root $\mathbb{L}^{1/2}$ of the Lefschetz object 
\begin{equation}
\label{eq:140}
  \mathbb{L}\coloneq H^{*}_c(\A^1,\Q) \in D^b(\MHM( \mathrm{pt})) \subset D^b( \MMHM( \mathrm{pt})), \qquad (\mathbb{L}^{1/2})^{\otimes 2} \simeq \mathbb{L},
\end{equation}
which is used to define the intersection complex \eqref{eq:117}.
Moreover, 
\begin{equation}
\label{eq:142}
  H^{*}(B\C^{\times})_{\vir}\coloneq \mathbb{L}^{1/2} \otimes H^{*}(B\C^{\times})
\end{equation}
is the virtual cohomology of $B\C^{\times}$ appearing in the cohomological integrality formula in Section \ref{sec:BPS decomposition}.

For stacks, $\MMHM( \mathfrak{X})$ is again defined by smooth descent. For $p\colon \mathfrak{X} \to X$ satisfying the approximation condition, one can also define $\mathcal{H}^n(p_{*}(-))$ on monodromic mixed Hodge modules as well.

\begin{proposition}[{\cite[Theorems 2.14 and 4.3]{Saito1990}}]\label{prop:proper-pushforward-vanishing-cycles}
Let $h\colon X \to B$ be a proper morphism of complex algebraic varieties, and let
\begin{equation}
\label{eq:174}
  \widetilde{f}\colon B\longrightarrow \A^1
\end{equation}
be a regular function. Set $f\coloneq \widetilde{f} \circ h$. Then for every $M \in D^b(\MHM(X))$ there is a natural isomorphism
\begin{equation}
\label{eq:19}
  Rh_*\phi^{\mathrm{mmhm}}_{f}(M)\simeq  \phi^{\mathrm{mmhm}}_{\widetilde{f}}(Rh_*M)
\end{equation}
in  $D^b(\MMHM(B))$.

After applying rationalisation, if
\begin{equation}
\label{eq:113}
  X_0\coloneq f^{-1}(0),\qquad B_0\coloneq \widetilde{f}^{-1}(0),
\end{equation}
and $h_0\colon X_0\to B_0$ is the induced proper morphism, then there is a natural
isomorphism
\begin{equation}
\label{eq:114}
  Rh_{0,*}\phi_f(\operatorname{rat}M) \simeq \phi_{\widetilde{f}}(Rh_*\operatorname{rat}M)
\end{equation}
in $D^b_c(B_0,\Q)$.
\end{proposition}

Write
\begin{equation}
\label{eq:143}
  D^{\geq, \mathrm{lf}} (\MMHM(X))
\end{equation}
for the full subcategory of locally finite complexes that are bounded below.\footnote{A complex $M$ is locally finite if the following conditions are met: 
\begin{enumerate}
\item\label{item:4} For all connected components $Z\subset X$, for all $n \in \Z$, the set $\{i \in \Z \mid \mathrm{gr}^W_n \mathcal{H}^i(M)|_Z \neq 0\}$ is finite. 
\item\label{item:5} There exists $n \in \Z$ such that for all $i$, $W_n \mathcal{H}^i(M)|_Z = 0$.
\end{enumerate}}
If $(X,m)$ is a commutative monoid scheme and $m\colon X \times X \to X$ is finite, define the convolution product of $M$ and $N$ to be
\begin{equation}
\label{eq:144}
  M\boxtimes_m N \coloneq m_{*}(M\boxtimes N).
\end{equation}
This gives a symmetric monoidal structure on $D^{\geq, \mathrm{lf}}(\MMHM(X))$ and hence symmetric product functors $\Sym^n_{\boxtimes_m}$.

\subsubsection{Intersection complexes and Maulik--Shen theorem}
Let $X$ be an irreducible complex variety and $j\colon X_{\mathrm{reg}} \hookrightarrow X$ the inclusion of the smooth locus. The ordinary perverse intersection complex is
\begin{equation}
\label{eq:116}
  \widetilde{ \IC}_X\coloneq j_{!*} \Q_{X_{\mathrm{reg}}}[\dim X].
\end{equation}
The normalised monodromic object is
\begin{equation}
\label{eq:117}
  \IC_X\coloneq \mathbb{L}^{-\dim X /2}\otimes \widetilde{\IC}_X.
\end{equation}
By abuse of notation, when no monodromic normalisation is being tracked, $\IC_X$ denotes the underlying perverse sheaf $\widetilde{\IC}_X$. 
If $\mathfrak{X}$ is a smooth complex Artin stack locally of finite type, we write
\begin{equation}
\label{eq:109}
  \IC_{\mathfrak{X}}\coloneq \Q_{\mathfrak{X}}[\dim \mathfrak{X}]
\end{equation}
or, in the mixed-Hodge-module setting, the corresponding shifted constant mixed Hodge module. For a disjoint union of smooth connected components, this convention is understood componentwise.

In the notation of Section \ref{sec:twisted higgs bundles}, let $g\colon \mathcal{C} \to U$ be the universal spectral curve over 
  \begin{equation}
  \label{eq:6}
    U \coloneq \{a \in B_Y\mid \textrm{the spectral curve } C_a \subset Y \textrm{ is smooth}\}.
  \end{equation}
The key input from Maulik--Shen \cite{MaulikShen2023}, is the following cohomological $\chi$-independence theorem for the Hitchin fibration. 

\begin{theorem}[{\cite[Theorem 0.4]{MaulikShen2023}}]\label{thm:ms}   
  Suppose that $\deg L>2g(C) -2$. Then there is an isomorphism 
  \begin{equation}
  \label{eq:4}
    h_{Y,*}  \IC_{M_Y^{\mathrm{ss}}(r,m)}  \cong \bigoplus_{i=0}^{2d} \IC\left(\bigwedge\nolimits^i \mathbf{R}^1 g_{*}\Q_{\mathcal{C}}\right)[-i+d],
  \end{equation}
  where $d$ is the genus of the fibers of the universal spectral curve. In particular, for fixed $r$ the direct image $h_{Y,*} \IC_{M_Y^{\mathrm{ss}}(r,m)}$ is independent of $m$. 
\end{theorem}

\subsubsection{BPS Invariants}

Let $p \colon \mathfrak{M} \to M$ be the good moduli space for a semistable moduli stack equipped with an oriented d-critical structure. Define
\begin{equation}
\label{eq:15}
  \phi_M^{\mathrm{mmhm}} \coloneq \mathcal{H}^1\left(p_{*}\phi_{\mathfrak{M}}^{\mathrm{mmhm}}\right) \in \MMHM(M).
\end{equation}
We write $\phi_M\coloneq\rat\left(\phi_M^{\mathrm{mmhm}}\right)$ for the underlying perverse sheaf.
\\

For a smooth quasi-projective Calabi--Yau threefold $X$ and a class $\beta$, the Maulik--Toda generalised Gopakumar--Vafa count starts with the perverse sheaf $\phi_{M_X^{\mathrm{ss}}(\beta,m)}$  on the moduli space of one-dimensional semistable sheaves. The \emph{BPS invariants} $n_{g,\beta,m}$ are defined via the following identity
\begin{equation}
\label{eq:110}
  \sum_{i\in \Z} \chi\left(\mathcal{H}^i\left(\HC_{*} \phi_{M_X(\beta,m)}\right)\right) y^i = \sum_{g\geq0} n_{g,\beta,m} \left(y^{1/2} + y^{-1/2}\right)^{2g},
\end{equation}
where $\HC$ is the Hilbert--Chow morphism of $M$. We consider the case $X = X_c$ and $(\pi\circ q)_{\#} \beta = r[C]$. The map $b\colon S_c \to B_Y$ constructed in Section \ref{sec:warped local curves} allows the Hilbert--Chow pushforward to be compared with the Hitchin-base pushforward; this is the mechanism used in Corollary \ref{cor:chi-independence}.

\subsection{Hochschild Homology and Dimensional Reduction}

Let $k$ be again an arbitrary field of characteristic $0$.

\subsubsection{Hochschild Homology}

The Hochschild homology complex of a $k$-linear stable $\infty$-category $\mathcal{C}$  is denoted by $\HH( \mathcal{C})$. The negative cyclic complex is 
\begin{equation}
\label{eq:153}
  \HC^-(\mathcal{C}) \coloneq \HH(\mathcal{C})^{\mathrm{S}^1}.
\end{equation}
There is a forgetful map 
\begin{equation}
\label{eq:154}
  (-)^{\natural} \colon \HC^-(\mathcal{C}) \to \HH(\mathcal{C})
\end{equation}
which sends a negative cyclic class to its underlying Hochschild class. We also use the connecting map induced by the mixed structure on Hochschild chains. In the degree conventions of \cite{BozecCalaqueScherotzke2024}, this gives for every $n \in \Z$, a map of spaces 
\begin{equation}
\label{eq:178}
 \delta \colon |\HH(\mathcal{C})[n+1]| \to |\HC^-(\mathcal{C})[n]|.
\end{equation}
Here the bars $|- |$ denote the underlying derived mapping space.
Thus, a Hochschild class 
\begin{equation}
\label{eq:179}
  c\colon k[-(n+1)] \to \HH(\mathcal{C})
\end{equation}
determines a negative cyclic class $\delta c$, whose underlying Hochschild class is denoted $(\delta c)^{\natural} \in \HH( \mathcal{C})$. A Hochschild class $c$ is \emph{exact} if there exists a Hochschild class $c' $ such that 
\begin{equation}
\label{eq:181}
  c  = (\delta c' )^{\natural}.
\end{equation}
For each $p\geq 0$, there are natural Chern character maps 
\begin{equation}
\label{eq:182}
  \operatorname{ch}^{(p)}_{ \mathcal{C}} \colon |\HH(\mathcal{C})[n]| \to \mathcal{A}^p( \mathfrak{M}_{\mathcal{C}}, n+p)
\end{equation}
from Hochschild classes to $p$-forms on $\mathfrak{M}_{\mathcal{C}}$, and closed Chern character maps 
\begin{equation}
\label{eq:183}
  \widetilde{ \operatorname{ch}}^{(p)}_{\mathcal{C}} \colon |\HC^-(\mathcal{C})[n]| \to \mathcal{A}^{p, \mathrm{cl}}( \mathfrak{M}_{\mathcal{C}}, n+p)
\end{equation}
from negative cyclic classes to closed $p$-forms. These maps are natural in $\mathcal{C}$ and the underlying $p$-form of $\widetilde{ \operatorname{ch}}^{(p)}_{\mathcal{C}}( \tilde{\alpha})$ is $\operatorname{ch}^{(p)}_{\mathcal{C}}(\tilde{\alpha}^{\natural})$.
Moreover, they are also compatible with the de Rham differential.
In other words, the following diagram commutes:
\begin{equation}
\label{eq:184}
\begin{tikzcd}
  {|}\HH(\mathcal{C})[n+1]{|} \arrow[r, "\operatorname{ch}^{(p-1)}_{\mathcal{C}}"] \arrow[d,"\delta"] & \mathcal{A}^{p-1} (\mathfrak{M}_{\mathcal{C}},n+p) \arrow[d,"\ddR"] \\
  {|}\HC^-(\mathcal{C})[n]{|} \arrow[r," \widetilde{ \operatorname{ch}}^{(p)}_{\mathcal{C}}"] \arrow[d,"(-)^{\natural}"]& \mathcal{A}^{p, \mathrm{cl}} (\mathfrak{M}_{\mathcal{C}}, n+p) \arrow[d]\\
  {|}\HH(\mathcal{C})[n]{|} \arrow[r, "\operatorname{ch}^{(p)}_{\mathcal{C}}"]  & \mathcal{A}^p (\mathfrak{M}_{\mathcal{C}},n+p) 
\end{tikzcd}
\end{equation}

\begin{example}
  The map 
  \begin{equation}
  \label{eq:186}
    \operatorname{ch}^{(0)}\colon |\HH(\mathcal{C})[n]| \to \mathcal{A}^0(\mathfrak{M}_{\mathcal{C}},n) = |\Gamma(\mathfrak{M}_{\mathcal{C}}, \O_{\mathfrak{M}_{\mathcal{C}}})[n]|
  \end{equation}
  assigns to a Hochschild class a function on the moduli stack of objects.
  Moreover, the map 
  \begin{equation}
  \label{eq:187}
     \operatorname{ch}^{(1)}\colon |\HH(\mathcal{C})[n]| \to \mathcal{A}^1(\mathfrak{M}_{\mathcal{C}},n+1)
  \end{equation}
  assigns to a Hochschild class a $1$-form on $\mathfrak{M}_{\mathcal{C}}$. If the Hochschild class is exact, say 
  \begin{equation}
  \label{eq:188}
    c = (\delta c' )^{\natural},
  \end{equation}
  for $c' \in |\HH(\mathcal{C})[n+1]|$,
  then the associated $1$-form is the de Rham differential of the function 
  \begin{equation}
  \label{eq:189}
    \operatorname{ch}^{(0)}_{\mathcal{C}}(c' ).
  \end{equation}
  That is
  \begin{equation}
  \label{eq:185}
      \operatorname{ch}^{(1)}\left( (\delta c' )^{\natural}\right) = \ddR \operatorname{ch}^{(0)}_{\mathcal{C}}(c' ).
  \end{equation}
\end{example}

\begin{remark}
  For a smooth variety $Y$, the Hochschild--Kostant--Rosenberg isomorphism identifies 
\begin{equation}
\label{eq:69}
  \HH_k (\QCoh(Y)) \xrightarrow{\sim} \bigoplus_{p-q=k} H^q (Y,\Omega^p_Y).
\end{equation}
In particular, $H^1(Y,\omega_Y)$ is contained in $\HH_{\dim Y - 1}(\QCoh(Y))$.
\end{remark}

\subsubsection{Dimensional Reduction and Deformed Calabi--Yau Categories}

We recall the two inputs on deformed Calabi--Yau completions needed to construct the global critical chart. Keller's deformed Calabi--Yau completion \cite{Keller2011} assigns to a stable category and a negative cyclic class a new Calabi--Yau category $\Pi_{d+1}(\mathcal{C},\tilde{c})$. Kinjo--Masuda use this construction for $\mathcal{C} = \QCoh(Y)$ and identify the resulting category with $\QCoh(X_c)$.

\begin{theorem}[{\cite[Corollary 6.19]{BozecCalaqueScherotzke2024}}]\label{thm:bcs}
  Let $d \in \Z$,  $\mathcal{C}$ a $k$-linear stable $\infty$-category of finite type, and $c\colon k[d-2] \to \HH(\mathcal{C})$ a Hochschild homology class. Let $\delta c$ be its image in $\HC^-( \mathcal{C})$ under \eqref{eq:178}.
  Then there exists  a natural function $f_c\colon \mathfrak{M}_{\mathcal{C}} \to \A^1[2-d]$ such that there exists  a natural equivalence  of $(1-d)$-shifted symplectic derived locally Artin stacks 
  \begin{equation}
  \label{eq:10}
    \mathfrak{M}_{\Pi_{d+1}(\mathcal{C},\delta c)} \simeq \Crit(f_c \colon \mathfrak{M}_{\mathcal{C}} \to \A^1[2-d]).
  \end{equation}
\end{theorem}

The second input is the geometric computation of the deformed Calabi--Yau completion for a torsor. Let $Y$ be a smooth variety of dimension $d$. The Connes exact sequence, together with the vanishing of negative cyclic homology above degree $d$, gives an exact sequence
\begin{equation}
\label{eq:214}
  0\longrightarrow \HC^{-}_{d-1}(\QCoh(Y)) \xrightarrow{(-)^\natural} \HH_{d-1}(\QCoh(Y)) \xrightarrow{B} \HH_d(\QCoh(Y)),
\end{equation}
see \cite[Equation (2.4)]{KinjoMasuda2024}. Under the Hochschild--Kostant--Rosenberg
isomorphism, $B$ is identified with the de Rham differential. Hence,
for
\begin{equation}
\label{eq:213}
  c\in H^1(Y,\omega_Y)=H^1(Y,\Omega_Y^d) \subset \HH_{d-1}(\QCoh(Y)),
\end{equation}
one has $B(c)=0$, since $\Omega_Y^{d+1}=0$. Therefore $c$ admits a
unique negative cyclic lift 
\begin{equation}
\label{eq:212}
  \tilde{c}\in\HC^{-}_{d-1}(\QCoh(Y)).
\end{equation}
Kinjo--Masuda identify the associated deformed Calabi--Yau completion with the category of quasi-coherent sheaves on the corresponding $\omega_Y$-torsor.

\begin{theorem}[{\cite[Theorem 4.3]{KinjoMasuda2024}}]\label{thm:km cy completion}
  Let $Y$ be a smooth variety of dimension $d$ over $k$ and let $c \in H^1(Y,\omega_Y)$ be  a class viewed as an element of $\HH_{d-1}(\QCoh(Y))$. The class $c$ lifts to a negative cyclic class $\tilde{c}\in\HC^{-}_{d-1}(\QCoh(Y))$. Let $X_c$ be the total space of the $\omega_Y$-torsor corresponding to $c$. 
  There is an equivalence of $k$-linear stable $\infty$-categories 
  \begin{equation}
  \label{eq:9}
    \Pi_{d+1} (\QCoh(Y), \tilde{c}) \cong \QCoh(X_c)
  \end{equation}
  over $\QCoh(Y)$ identifying the natural left $(d+1)$-Calabi--Yau structures. 
\end{theorem}

\begin{corollary}\label{cor:critical locus}
  Let $Y$ be a smooth variety of dimension $d$ and $c \in H^1(Y,\omega_Y)$ an exact class, say $c = \ddR c' $ for some $c' \in H^1(Y,\Omega_Y^{d-1})$. Let $X_c$ be the total space of the $\omega_Y$-torsor corresponding to $c$. There exists a function $f_c\colon \mathfrak{M}_Y \to \A^1[2-d]$ such that there exists an equivalence of $(1-d)$-shifted symplectic derived locally Artin stacks 
  \begin{equation}
  \label{eq:12}
    \mathfrak{M}_{X_c} \simeq \Crit(f_c\colon \mathfrak{M}_Y \to \A^1[2-d])
  \end{equation}
  over $\mathfrak{M}_Y$.
\end{corollary}

\begin{proof}
  Let
  \begin{equation}
  \label{eq:192}
    \tilde{c} \in \HC^-_{d-1}(\QCoh(Y))
  \end{equation}
  be the negative cyclic lift of $c \in \HH_{d-1}(\QCoh(Y))$ appearing in Theorem \ref{thm:km cy completion}. By compatibility of the Hochschild--Kostant--Rosenberg isomorphism with the mixed differential, 
  \begin{equation}
  \label{eq:217}
    (\delta c')^\natural=d_{\mathrm{dR}}c'=c.
  \end{equation}
  Thus $\delta c'$ is also a negative cyclic lift of $c$. Since 
  \begin{equation}
  \label{eq:218}
    \HC^-_{d-1}(\QCoh(Y)) \to \HH_{d-1}(\QCoh(Y))  
  \end{equation}
  is injective, we obtain 
  \begin{equation}
  \label{eq:219}
    \tilde{c}=\delta c'.
  \end{equation}
  Then
  \begin{equation}
  \label{eq:172}
    \Pi_{d+1}(\QCoh(Y), \tilde{c}) \cong \QCoh(X_c).
  \end{equation}
  Applying Theorem \ref{thm:bcs} to the Hochschild class $c'$, we obtain a function
  \begin{equation}
   \label{eq:220}
    f_c\coloneq \ch^{(0)}(c' ) \colon \mathfrak{M}_Y \to \A^1[2-d]
  \end{equation}
  such that
  \begin{equation}
  \label{eq:173}
    \mathfrak{M}_{X_c} \simeq \mathfrak{M}_{\Pi_{d+1}(\QCoh(Y), \tilde{c})} \simeq \Crit(f_c\colon \mathfrak{M}_Y \to \A^1[2-d]).
  \end{equation}
\end{proof}

For a surface $Y = \tot_C(L)$, exactness is automatic by Lemma \ref{lem:cohomology-primitive}. We show in Section \ref{sec:proof A} that the resulting function $f_c$ is the trace pairing of the primitive of $c \in H^1(Y,\omega_Y)$ if $\deg L > 2g(C) -2$. In that case, after identifying $\mathfrak{M}_Y^{\mathrm{ss}}(r,m)$ with semistable $L$-valued Higgs bundles, we can write
\begin{equation}
\label{eq:74}
  f_c([E,\theta]) = \sum_{k\geq 1} \frac{1}{k} c_k(\tr(\theta^k)).
\end{equation}
Here $c = \sum_{k\geq1 } \iota_k(c_k)$ under \eqref{eq:H1-decomp}.
The additivity of the trace under direct sum is also the reason this function is compatible with the BPS decomposition theorem (see Section \ref{sec:BPS decomposition}).

\section{Global Critical Charts}\label{sec:global critical chart}

Let $k$ be a field of characteristic $0$.

\subsection{Global Critical Chart for Total Spaces of Line Bundles}\label{sec:proof A}
We prove Theorem \ref{thm:A} (= Theorem \ref{thm:global critical chart}).

\begin{lemma}\label{lem:cohomology-primitive}
  Let $C$ be a smooth projective curve over $k$, $L$ a line bundle on $C$, and $Y \coloneq \tot_C(L)$ its total space with projection $\pi\colon Y \to C$. 
  There is a natural decomposition
\begin{equation}\label{eq:H1-decomp}
        H^1(Y,\omega_Y)  \cong  \bigoplus_{k\geq 1}H^1(C,\omega_C\otimes L^{-k}).
\end{equation}
If $c=\sum_{k\geq 1}\iota_k(c_k)$ is the corresponding expression of a class $c\in H^1(Y,\omega_Y)$, then there exists
\begin{equation}
\label{eq:88}
  c'=\sum_{k\geq 1}\frac{1}{k}\tau_k(c_k)\in H^1(Y,\Omega_Y^1)
\end{equation}
such that $\ddR c'=c$.
\end{lemma}

\begin{proof}
Since $\pi$ is affine and $\omega_Y\cong \pi^*(\omega_C\otimes L^\vee)$, we have
\begin{equation}
\label{eq:87}
  \pi_*\omega_Y    \cong  (\omega_C\otimes L^\vee)\otimes \Sym L^\vee  \cong   \bigoplus_{k\geq 1}\omega_C\otimes L^{-k},
\end{equation}
which gives \eqref{eq:H1-decomp}. Let $\lambda\in H^0(Y,\pi^*L)$ be the tautological section. The map
\begin{equation}
\label{eq:86}
  \tau_k:H^1(C,\omega_C\otimes L^{-k})\longrightarrow H^1(Y,\Omega^1_Y),  \qquad    \alpha\longmapsto \pi^*\alpha\cdot \lambda^k,
\end{equation} 
uses the natural inclusion $\pi^*\omega_C\hookrightarrow \Omega^1_Y$ from the short exact sequence 
\begin{equation}
\label{eq:91}
  0 \to \pi^{*} \omega_C \to \Omega^1_Y \to \Omega^1_{Y/C} \to 0.
\end{equation}
For ease of notation, we denote the zero map by $\tau_0\colon H^1(C,\omega_C) \to H^1(Y,\Omega^1_Y)$.
In a local trivialization of $L$, with fibre coordinate $t$, the de Rham differential satisfies
\begin{equation}
\label{eq:84}
  \ddR(\pi^*\alpha\cdot t^k)=k\,\pi^*\alpha\cdot t^{k-1}dt,
\end{equation}
and the latter is precisely the $k$-th summand in $H^1(Y,\omega_Y)$:
\begin{equation}
  \label{eq:500}
  \begin{tikzcd}
    \bigoplus_{k\geq 0} H^1(C,\omega_C\otimes L^{-k}) \arrow[rr, "\sum_{k\geq 0}\tau_k"] \arrow[d,"\cdot k"]&& H^1(Y,\Omega^1_Y) \arrow[d,"\ddR"] \\
    \bigoplus_{k\geq 1} H^1(C,\omega_C \otimes L^{-k}) \arrow[rr, "\sum_{k\geq1}\iota_k", "\sim" swap]& & H^1(Y,\omega_Y).
  \end{tikzcd}
\end{equation}
Hence,
\begin{equation}
\label{eq:85}
  \ddR(\tau_k(\alpha))=k\iota_k(\alpha).
\end{equation}
\end{proof}

\begin{theorem}[= Theorem \ref{thm:A}]\label{thm:global critical chart}
  Let $C$ be a smooth projective curve over $k$, let $L$ be a line bundle on $C$, and let 
  \begin{equation}
  \label{eq:64}
    \pi\colon Y:=\tot_C(L)\longrightarrow C
  \end{equation}
  be the total space of $L$. Fix a class $c\in H^1(Y,\omega_Y)$ and denote by $q\colon X_c\to Y$ the total space of the $\omega_Y$-torsor associated to it.
  There exists a regular function 
  \begin{equation}
  \label{eq:65}
     f_c\colon \mathfrak{M}_Y \to \A^1
  \end{equation}
  on the derived moduli stack $\mathfrak{M}_Y$ of compactly supported perfect complexes on $Y$ such that there exists a natural equivalence of $(-1)$-shifted symplectic derived Artin stacks  
  \begin{equation}
  \label{eq:2}
    \mathfrak{M}_{X_c}\simeq \Crit (f_c)
  \end{equation}
  over $\mathfrak{M}_Y$. 
\end{theorem}

\begin{proof}
  By Lemma \ref{lem:cohomology-primitive}, every class $c\in H^1(Y,\omega_Y)$ admits a de Rham primitive $c'\in H^1(Y,\Omega^1_Y)$ with $\ddR c'=c$. Viewing $c$ as a Hochschild class 
  \begin{equation}
  \label{eq:141}
    c \in H^1(Y,\omega_Y) \subset \HH_1(\QCoh(Y)),
  \end{equation}
  via the Hochschild--Kostant--Rosenberg isomorphism, the class $\delta c'$ gives the negative cyclic lift $\tilde{c}$ required in Corollary \ref{cor:critical locus}. 
  Therefore, there is a function
  \begin{equation}
  \label{eq:77}
    f_c\colon \Mf_Y\longrightarrow \A^1
  \end{equation}
  and an equivalence of $(-1)$-shifted symplectic derived Artin stacks
  \begin{equation}
  \label{eq:171}
    \Mf_{X_c}\simeq \Crit(f_c)
  \end{equation}
  over $\mathfrak{M}_Y$.    
\end{proof}

\begin{corollary}\label{corollary:explicit function}
  Consider the setting above under the additional assumption that $\deg L>2g(C)-2$. Then, on the semistable locus identified with $L$-valued Higgs bundles, the function is given by 
  \begin{equation}
  \label{eq:5}
    f_c\left( [(E,\theta)]\right) = \sum_{k\geq 1} \frac{c_k(\tr(\theta^k))}{k},
  \end{equation}
  where $\theta^k \in H^0(C,\End(E) \otimes L^{\otimes k})$ is the $k$-th power of the Higgs field and 
  \begin{equation}
  \label{eq:169}
    \tr(\theta^k) \in H^0(C,L^{\otimes k}).
  \end{equation}
  Here $c = \sum_{k\geq1} \iota_k(c_k)$ is a finite sum, where 
  \begin{equation}
  \label{eq:82}
    c_k \in H^1(C,\omega_C\otimes L^{-k}) \cong  H^0(C,L^{\otimes k})^{\vee}
  \end{equation}
  and 
  \begin{equation}
  \label{eq:61}
    \iota_k \colon H^1(C,\omega_C\otimes L^{-k}) \hookrightarrow H^1(C,\pi_{*}\omega_Y) \cong H^1(Y,\omega_Y)
  \end{equation}
  denotes the inclusion. 
\end{corollary}

\begin{proof}
  The function is the trace of the primitive,
\begin{equation*}
        f_c=\ch^{(0)}(c')
        =\ch^{(0)}\!\left(\sum_{k\geq 1}\frac{1}{k}\tau_k(c_k)\right),
\end{equation*}
where $c=\sum_{k\geq 1}\iota_k(c_k)$ under \eqref{eq:H1-decomp}. The final formula \eqref{eq:5} follows from the same calculation as in \cite[Proposition 5.7]{KinjoMasuda2024}: the class $\tau_k(c_k)$ pairs with the $k$-th power of the Higgs field by $c_k(\tr(\theta^k))$, and the primitive $c'$ contributes the factor $1/k$.
\end{proof}

\subsection{Semistability}\label{sec:semistability}

\begin{proposition}\label{prop:semistable}
  In the setting above, under the additional assumption that $\deg L > 2g(C)-2$, there is an equivalence of $(-1)$-shifted symplectic derived Artin stacks
  \begin{equation}
  \label{eq:3}
    \mathfrak{M}_{X_c}^{\mathrm{ss}}(r,m) \simeq \Crit \left(f_c|_{\mathfrak{M}_Y^{\mathrm{ss}}(r,m)}\right).
  \end{equation}
\end{proposition}

\begin{proof}
  Let $A_c \coloneq q_{*}\O_{X_c}$. Let $F$ be a compactly supported pure $1$-dimensional sheaf on $X_c$.
  As $q$ is affine, $q_{*}$ identifies coherent sheaves on $X_c$ with coherent $A_c$-modules on $Y$: 
  \begin{equation}\label{eq:80}
  \left\{
  \begin{array}{c}
    \textrm{coherent sheaves} \\ \textrm{ on } X_c
  \end{array}\right\} \longleftrightarrow
  \left\{  \begin{array}{c}
    \textrm{coherent } A_c \textrm{-modules} \\ \textrm{ on } Y
  \end{array}\right\}
  \end{equation}
  If $F$ is a compactly supported pure $1$-dimensional sheaf on $X_c$, write $G\coloneq q_{*}F$ for the corresponding $A_c$-module. As $q$ is affine, $\chi(F) = \chi(G)$. Moreover, the fundamental cycle of $G$ is $q_{\#}[F]$. Hence, the slope is preserved under \eqref{eq:80}. Moreover, subsheaves of $F$ are exactly $A_c$-submodules of $G$.
  By Theorem \ref{thm:global critical chart}, we only have to check that the semistable loci match under $q_{*}$, i.e., $F$ is semistable if and only if $G \coloneq q_{*}F$ is semistable.

  If $F$ is not semistable, it has a destabilising subsheaf $F'\subset F $. The pushforward $q_{*}F'  \subset q_{*}F = G$ of the destabilising subsheaf is an $A_c$-submodule of $G$. In particular, it is an $\O_Y$-subsheaf of $G$. Since slope is preserved under the identification \eqref{eq:80}, $q_{*}F' $ destabilises $G$. 
  
  Conversely, assume that $G$ is not semistable as an $\O_Y$-module. Let
  \begin{equation}
  \label{eq:11}
    0 = G_0 \subset  G_1\subset G_2 \subset \ldots \subset G_n = G,
  \end{equation}
  be its Harder--Narasimhan filtration. Let $Q\coloneq G/G_1$. Then
  \begin{align}
  \label{eq:13}
    \mu(G_1) > \mu(G_2/G_1) = \mu_{\mathrm{max}} (Q)     \implies & \mu(G_1) > \mu_{\mathrm{max}}(Q\otimes \omega_Y)\\ 
    \label{eq:16}
    \implies & \Hom(G_1, Q\otimes \omega_Y) \simeq \Hom(\omega_Y^{\vee} \otimes G_1, Q) = 0.
  \end{align}
  The implication \eqref{eq:13} holds because $\deg L > 2g(C) -2$. Indeed, because $\omega_Y\simeq \pi^{*}(\omega_C\otimes L^{\vee})$ and $\deg L > 2g(C) - 2$,  $\deg (\omega_C\otimes L^{\vee}) <0$. Thus,
  \begin{equation}
  \label{eq:20}
  \mu_{\mathrm{max}} (Q\otimes \omega_Y) = \mu_{\mathrm{max}} (Q) + \deg (\omega_C\otimes L^{\vee}) < \mu_{\mathrm{max}}(Q).
  \end{equation}
  We now show that $G_1$ is an  $A_c$-submodule of $G$. It is enough to prove that
  \begin{equation}
  \label{eq:17}
    A_c \otimes G_1 \to G \to Q
  \end{equation}
  is $0$. The degree filtration on $A_c$ satisfies
  \begin{equation}
  \label{eq:18}
    F_0A_c=\O_Y\subset F_1A_c\subset\ldots\subset F_kA_c\subset\ldots,  \qquad \gr_k^F A_c\simeq\omega_Y^{-k},
  \end{equation}
  and $F_1A_c$ generates $A_c$ as an $\O_Y$-algebra. In particular, $\bigcup_{k\geq0}F_kA_c=A_c$ and $F_1A_c/F_0A_c\simeq\omega_Y^\vee$. It therefore suffices to show that $F_1A_c\cdot G_1\subset G_1$. Let 
  \begin{equation}
  \label{eq:81}
    \beta\colon F_1 A_c \otimes G_1 \to G \to  Q
  \end{equation}
  denote the composition. 
  As $G_1$ is an $\O_Y$-submodule of $G$, $\O_Y\cdot G_1 \subset G_1$, so $\beta|_{\O_Y\otimes G_1} \equiv 0$.  By the right-exact sequence 
  \begin{equation}
  \label{eq:21}
    \O_Y\otimes G_1 \to F_1A_c \otimes G_1 \to \omega_Y^{\vee} \otimes G_1 \to 0,
  \end{equation}
  the morphism $\beta \colon F_1 A_c \otimes G_1 \to Q$ descends to a map $\overline{\beta}\colon \omega_Y^{\vee} \otimes G_1 \to Q$.
  By \eqref{eq:16}, any morphism $\omega_Y^{\vee}\otimes G_1 \to Q$ vanishes, so $F_1A_c \cdot G_1 \subset G_1$. Since $F_1A_c$ generates $A_c$, $A_c \cdot G_1 \subset G_1$. Thus, $G_1$ is an $A_c$-submodule of $G$ and therefore corresponds to a subsheaf  of $F$ on $X_c$. Since $\mu(G_1) > \mu(G)$, the subsheaf of $F$ corresponding to $G_1$ destabilises $F$.

  Because the equivalence in Theorem \ref{thm:global critical chart} is over $\mathfrak{M}_Y$, 
  \begin{equation}
  \label{eq:199}
    \mathfrak{M}_{X_c} \times_{\mathfrak{M}_Y} \mathfrak{M}_Y^{\mathrm{ss}}(r,m) \simeq \Crit(f_c) \times_{\mathfrak{M}_Y} \mathfrak{M}_Y^{\mathrm{ss}}(r,m).
  \end{equation}
  The left-hand side is precisely $\mathfrak{M}_{X_c}^{\mathrm{ss}}(r,m)$ and the right-hand side is $\Crit\left(f_c|_{ \mathfrak{M}_Y^{\mathrm{ss}}(r,m)}\right)$. 
\end{proof}

\begin{remark}\label{rem:factor-through-hitchin}
  On the fixed semistable sheaf locus $\Mf_Y^{\mathrm{ss}}(r,m)$, the restriction of $f_c$ factors through the Hitchin base. Thus,
\begin{equation}
\label{eq:60}
\begin{tikzcd}
  \mathfrak{M}_{Y}^{\mathrm{ss}}(r,m)\arrow[dr,"\tilde{h}_Y" swap] \arrow[rr, "f_c"] && \A^1 \\
  & B_Y \arrow[ur, "\tilde{f}_c" swap] &
\end{tikzcd}
\end{equation}
with
\begin{equation}
\label{eq:63}
  \widetilde f_c\left((\tr(\theta^k))_{k=1}^r\right)        =\sum_{k\geq 1}\frac{c_k(\tr(\theta^k))}{k}.
\end{equation}
\end{remark}

\begin{proposition}\label{prop:naturally oriented}
  Let $k = \C$. The stack $\mathfrak{M}_{X_c}^{\mathrm{ss}}(r,m)$ has a naturally oriented d-critical structure. 
\end{proposition}

\begin{proof}
  By Proposition \ref{prop:semistable}, the d-critical structure on  $\mathfrak{M}_{X_c}^{\mathrm{ss}}(r,m)$ coming from its $(-1)$-shifted symplectic structure agrees with the d-critical structure given by the classical truncation of the critical locus description. 
  By \cite[Lemma A.10]{KinjoKoseki2026}, if $Z = \Crit(f)$ for a function $f$ on a smooth Artin stack $U$, then there is a natural orientation 
  \begin{equation}
  \label{eq:103}
    o\colon K^{\otimes 2}_U |_{Z_{\red}} \xlongrightarrow{\sim} K^{\vir}_{Z,s}.
  \end{equation}
  Applying this to $Z = \mathfrak{M}_{X_c}^{\mathrm{ss}}(r,m)$ and $U = \mathfrak{M}_Y^{\mathrm{ss}}(r,m)$, we get a natural orientation on $\mathfrak{M}_{X_c}^{\mathrm{ss}}(r,m)$. 
\end{proof}

\subsection{Global Critical Chart for Curve Fibration over $\A^1$}
Let $k = \C$ and let $S$ be a smooth complex algebraic surface with a proper surjective morphism $\pi \colon S \to \A^1$  with connected one-dimensional fibres. The aim of this section is to show that the  morphism $H^1(S, \Omega^1_S) \to H^1(S,\omega_S)$ induced by the de Rham differential is surjective and therefore, the moduli space of sheaves is the critical locus of a function.

\begin{lemma}\label{lemma:H3 vanishes}
  For $S$ as above, $H^3_{\mathrm{dR}}(S) = 0$.
\end{lemma}

\begin{proof}
  By the Leray spectral sequence for $\pi\colon S \to \A^1$, it suffices to show that all $E_2^{p,q} = 0$ for $p+q = 3$ as this implies that
  \begin{equation}
  \label{eq:537}
    H^3_{\mathrm{dR}}(S) \simeq H^3(S,\C) = 0.
  \end{equation}
  Because the fibres of $\pi$ are curves, $R^q \pi_{*} \C = 0$ for all $q>2$. Moreover, as the fibres are connected, $R^0\pi_{*}\C \cong \C_{\A^1}$, and hence $H^3(\A^1(\C), R^0\pi_{*}\C) = 0$.
  Artin's vanishing theorem  implies that $H^2(\A^1(\C), R^1\pi_{*}\C) = 0$.
  Therefore,  we are left to show that 
  \begin{equation}
  \label{eq:538}
     H^1(\A^1(\C), R^2\pi_{*}\C) = 0.
  \end{equation}
  Over the smooth locus $U \subset \A^1(\C)$, 
  \begin{equation}
  \label{eq:541}
   \forall t \in U: \qquad    H^2(S_t,\C) \simeq \C .
  \end{equation}
  Therefore, 
  \begin{equation}
  \label{eq:546}
    R^2\pi_{*} \C|_U\simeq \C_U.
  \end{equation}
  Over a singular fibre $S_{t_0}$, $H^2(S_{t_0},\C)$ is the direct sum of copies of $\C$, one copy for each irreducible component of $S_{t_0}$. There is a short exact sequence 
  \begin{equation}
  \label{eq:549}
    0 \to \mathcal{K} \to R^2\pi_{*}\C \xrightarrow{\tr} \C_{\A^1} \to 0,
  \end{equation}
  where the kernel $\mathcal{K}$ is a skyscraper sheaf supported on the reducible singular fibres. As $\A^1(\C)$ is contractible, 
  \begin{equation}
  \label{eq:550}
    H^1( \A^1(\C), \C_{\A^1}) = 0.
  \end{equation}
  And as $\mathcal{K}$ is a skyscraper sheaf, 
  \begin{equation}
  \label{eq:551}
    H^1(\A^1(\C), \mathcal{K}) = 0.
  \end{equation}
  Therefore, by the long exact sequence in cohomology, 
  \begin{equation}
  \label{eq:552}
    H^1(\A^1(\C), R^2\pi_{*}\C) = 0.
  \end{equation}
\end{proof}

\begin{lemma}\label{lemma:H2 vanishes}
  In the setting above, for all coherent sheaves $\mathcal{F}$ on $S$, 
  \begin{equation}
  \label{eq:66}
    H^2(S, \mathcal{F}) = 0.
  \end{equation}
\end{lemma}

\begin{proof}
 Because $\pi\colon S \to \A^1$ has relative dimension $1$ and $\A^1$ is affine, every term of total degree $2$ in the Leray spectral sequence vanishes.  Thus, $H^2(S, \mathcal{F}) = 0$. 
\end{proof}

\begin{theorem}\label{thm:primitive curve fibration}
  The map $H^1(S, \Omega^1_S) \to H^1(S,\omega_S)$ induced by the de Rham differential is surjective. 
\end{theorem}

\begin{proof}
  We use the Hodge--de Rham spectral sequence. Note that 
  \begin{equation}
  \label{eq:558}
    E_2^{2,1} = \coker \left(d_{\mathrm{dR}}\colon H^1(S,\Omega^1_S) \to H^1(S,\omega_S)\right).
  \end{equation}
  Thus, we need to show that $E_{2}^{2,1} =0$. We do this in three steps: First, we show that the $(2,1)$-piece stabilises at the $E_2$-page. Second, we show that $E_\infty^{2,1} = H^3_{\mathrm{dR}}(S)$. Lastly, we invoke Lemma \ref{lemma:H3 vanishes} to show that $H^3_{\mathrm{dR}}(S) = 0$.
  
  There are no outgoing differentials from $E_2^{2,1}$, because they would land in $E_r^{2+r,2-r}$ and for $r\geq 2$ those vanish. For $r = 2$ the only incoming differential is $E_2^{0,2} \to E_2^{2,1}$. By Lemma \ref{lemma:H2 vanishes}, $E_2^{0,2} = H^2(S,\O) = 0$. For $r\geq 3$, $E_r^{2-r,r}$ is outside the first quadrant and therefore equal to $0$. Hence, there are no non-zero incoming differentials and $E_2^{2,1} = E_{\infty}^{2,1}$.
  
  As $H^n_{\mathrm{dR}}(S)$ has a filtration 
  \begin{equation}
  \label{eq:560}
    H^n_{\mathrm{dR}}(S) = F^0 H^n_{\mathrm{dR}}(S) \supseteq F^1H^n_{\mathrm{dR}}(S) \supseteq F^2H^n_{\mathrm{dR}}(S) \supseteq \ldots 
  \end{equation}
  such that
  \begin{equation}
  \label{eq:559}
    \gr^p H^n_{\mathrm{dR}}(S) \coloneq  F^p H^n_{\mathrm{dR}}(S) / F^{p+1}H^n_{\mathrm{dR}}(S) \simeq E_\infty^{p,n-p},
  \end{equation}
  to show $H^3_{\mathrm{dR}}(S) = E_{\infty}^{2,1}$, we have to show that all graded pieces vanish except for $\gr^2 H^3_{\mathrm{dR}}(S) = E_{\infty}^{2,1}$. Because $S$ is a surface, $E_{\infty}^{3,0} = 0$. By Lemma \ref{lemma:H2 vanishes}, $E_{\infty}^{1,2} = 0$. Moreover, by Grothendieck vanishing $E_{\infty}^{0,3}=0$. Hence, 
  \begin{equation}
  \label{eq:562}
    F^3H^3_{\mathrm{dR}}(S) = 0 \qquad \text{and} \qquad F^0 H^3_{\mathrm{dR}}(S)= F^1H^3_{\mathrm{dR}}(S) = F^2H^3_{\mathrm{dR}}(S) = H^3_{\mathrm{dR}}(S).
  \end{equation}
  Thus,
  \begin{equation}
  \label{eq:561}
    E_{\infty}^{2,1} = \gr^2 H^3_{\mathrm{dR}}(S) = H^3_{\mathrm{dR}}(S).
  \end{equation}
  
  By Lemma \ref{lemma:H3 vanishes}, $H^3_{\mathrm{dR}}(S) = 0$. Therefore, \eqref{eq:558} vanishes.
\end{proof}

\begin{corollary}[= Theorem \ref{thm:C}]\label{corollary:global critical chart curve fibration}
  Let $S$ be a smooth complex surface and $\pi\colon S \to \A^1$ a proper surjective morphism with connected one-dimensional fibres. Fix a class $c \in H^1(S,\omega_S)$ and denote by $q\colon X_c \to S$ the total space of the $\omega_S$-torsor associated to it. There exists a function 
  \begin{equation}
  \label{eq:601}
    f_c\colon \mathfrak{M}_S \to \A^1
  \end{equation}
  on the derived moduli stack $\mathfrak{M}_S$ of compactly supported perfect complexes on $S$ such that there exists a natural equivalence of $(-1)$-shifted symplectic derived Artin stacks 
  \begin{equation}
  \label{eq:180}
    \mathfrak{M}_{X_c} \simeq \Crit(f_c).
  \end{equation}
\end{corollary}

\begin{proof}
  Analogous to the proof of Theorem \ref{thm:A} (= Theorem \ref{thm:global critical chart}) invoking Theorem \ref{thm:primitive curve fibration} in place of Lemma \ref{lem:cohomology-primitive}.
\end{proof}

\section{$\chi$-Independence}\label{sec:chi-independence}

\subsection{One-Dimensional Sheaves on Local Surfaces}
Let $k$ be a field of characteristic $0$ and let $C$ be a smooth projective curve over $k$. 
Let $Y \coloneq \tot_C(L)$ be the total space of a line bundle $L$ of degree $\deg L >2g(C) -2$ on $C$.

\begin{proposition}\label{prop:existence of curve in Xc}
  If $c \in \bigoplus_{k\geq 2} H^1(C,\omega_C\otimes L^{-k})$, then the total space $X_c$ of the torsor corresponding to $c$ contains a proper integral curve. 
\end{proposition}

\begin{proof}
  Pulling back the $\omega_Y$-torsor $X_c$ along the zero section $s\colon C \hookrightarrow Y$, we get a $s^{*}\omega_Y$-torsor $s^{*}X_c$. Its class is the pullback $s^{*}c \in H^1(C,s^{*}\omega_Y) \cong H^1(C, \omega_C\otimes L^{\vee})$. Under the identification
  \begin{equation}
  \label{eq:578}
    H^1(Y,\omega_Y) \cong \bigoplus_{k\geq 1}  H^1(C,\omega_C\otimes L^{-k}),
  \end{equation}
  the map
  \begin{equation}
  \label{eq:507}
    s^{*}\colon    H^1(Y,\omega_Y) \to H^1(C,\omega_C\otimes L^{\vee})
  \end{equation}
  is the projection onto the first component. 
  If $c$ has no component in $H^1(C,\omega_C\otimes L^{\vee})$, $s^{*}c = 0$. Therefore, the $s^{*}\omega_Y$-torsor $s^{*}X_c$  associated to $s^{*}c$ is trivial. Hence, it has a global section $\sigma\colon C \to s^{*}X_c$. This gives a lift of $s$ to $X_c$:
  \begin{equation}
  \label{eq:508}
    \tilde{s} \colon C \xrightarrow{\sigma} s^{*}X_c \to  X_c
  \end{equation}
  Since $s^{*} X_c \to X_c$ is the base-change of a closed immersion, it is itself a closed immersion. Since $\sigma$ is a section of a separated morphism, it is a closed immersion. Therefore, $\tilde{s}$ is a closed immersion.  Thus, its image is a proper curve in $X_c$. 
\end{proof}

\begin{corollary}\label{corollary:existence curve class}
  Let $r \in \Z_{>0}$.  If $c \in \bigoplus_{k\geq2} H^1(C,\omega_C\otimes L^{-k})$, $X_c$ contains an effective curve class $[\Gamma]$ such that $(\pi\circ q)_{\#} [\Gamma] = r[C]$.  
\end{corollary}

\begin{proof}
  By the proof of Proposition \ref{prop:existence of curve in Xc}, there is a section $\tilde{s} \colon C \to X_c$ whose image $\Gamma_0 = \tilde{s}(C)$ is a curve in $X_c$. We claim that 
  \begin{equation}
  \label{eq:32}
    (\pi\circ q)_{\#} [\Gamma_0] = [C]. 
  \end{equation}
  Indeed, because $\tilde{s} \colon C \to X_c$ is a lift of $s\colon C\to Y$, $q \circ \tilde{s} = s$.
  Therefore, 
  \begin{equation}
  \label{eq:79}
    \pi \circ s = \pi \circ q \circ \tilde{s}.
  \end{equation}
  Then, 
  \begin{equation}
  \label{eq:89}
    (\pi\circ q)|_{\Gamma_0} = \pi \circ s \circ \tilde{s}^{-1}|_{\Gamma_0} = \id_C.
  \end{equation}
  As $\tilde{s}(C) = \Gamma_0$, this implies \eqref{eq:32}. Finally, take $[\Gamma] \coloneq r[\Gamma_0]$. 
\end{proof}

\begin{proposition}\label{prop:unique lift}
  Let $i\colon D \hookrightarrow Y$ be an integral proper curve.
   There is at most one integral proper curve $\Gamma\subset X_c$ such that $q(\Gamma) = D$. That is, a curve in $X_c$ is uniquely determined by its image in $Y$. 
\end{proposition}

\begin{proof}
  Assume that there exists an integral projective curve $\Gamma \subset X_c$ such that $q(\Gamma) = D$. Let $\nu\colon \tilde{D} \to D$ be the normalisation of $D$.
  We show that any such curve $\Gamma$ is the image of $\tilde{D}$ under a section $s_D$ of a torsor $T_D$ associated to $D \subset Y$.
  
 Since $\omega_Y\simeq \pi^{*}(\omega_C\otimes L^{\vee})$ and $\tilde{D} \xrightarrow{\nu} D \xrightarrow{i} Y \xrightarrow{\pi} C$ is finite of positive degree, 
  \begin{equation}
  \label{eq:27}
    \deg (\nu^{*}i^{*}\omega_Y) = \deg (\tilde{D} \to C) \cdot (2g(C) - 2 - \deg L) <0.
  \end{equation}
  Hence, 
  \begin{equation}
  \label{eq:31}
    H^0(\tilde{D}, \nu^{*}i^{*}\omega_Y) = 0.
  \end{equation}
  Let $T_D \coloneq \tilde{D} \times_Y X_c$ be the $\nu^{*}i^{*}\omega_Y$-torsor given by the fibre product 
  \begin{equation}
  \label{eq:24}
  \begin{tikzcd}
    T_D \arrow[d] \arrow[r] & X_c \arrow[d,"q"] \\
    \tilde{D} \arrow[r,"i\circ \nu" swap] &  Y.
  \end{tikzcd}
\end{equation}
  By \eqref{eq:31}, $T_D \to \tilde{D}$ has at most one section.

  We now show that $T_D$ in fact has a section. Let $\mu\colon \tilde{\Gamma} \to \Gamma$ be the normalisation of $\Gamma$. The map $\tilde{\Gamma} \to D$ factors uniquely through $\nu$, 
  \begin{equation}
  \label{eq:122}
  \begin{tikzcd}
    \tilde{\Gamma} \arrow[d,"\mu"] \arrow[r,"f"] & \tilde{D} \arrow[d,"\nu"] \\
    \Gamma \arrow[r,"q|_{\Gamma}"] & D.
  \end{tikzcd}
  \end{equation}
  By the universal property of the fibre product, 
  \begin{equation}
  \label{eq:34}
  \begin{tikzcd}
    \tilde{\Gamma} \arrow[dr,dashed] \arrow[drr, bend left = 30] \arrow[ddr, bend right=30] &&\\
    & T_D \arrow[r] \arrow[d] & X_c  \arrow[d,"q"] \\
    & \tilde{D} \arrow[r,"i\circ \nu"] &Y
  \end{tikzcd}
  \end{equation}
  there is a map $\tilde{\Gamma} \to T_D$. 
  Again by the universal property of the fibre product, this induces a section $t_{\Gamma}$ of $f^{*}T_D$:
  \begin{equation}
  \label{eq:29}
  \begin{tikzcd}
    \tilde{\Gamma} \arrow[dr,dashed,"t_{\Gamma}"] \arrow[drr, bend left = 30] \arrow[ddr, bend right=30, "\id_{\tilde{\Gamma}}" swap] && \\
    & f^{*}T_D \arrow[r] \arrow[d] & T_D \arrow[d] \\
    & \tilde{\Gamma} \arrow[r,"f"] & \tilde{D}.
  \end{tikzcd}
  \end{equation}
  Therefore, $f^{*}T_D$ is the trivial $f^{*}\nu^{*}i^{*}\omega_Y$-torsor. Note that by 
  \begin{equation}
  \label{eq:129}
    \deg (f^{*}\nu^{*}i^{*}\omega_Y) = \deg (f) \cdot \deg (\nu^{*}i^{*}\omega_Y) <0, 
  \end{equation}
  $H^0(\tilde{\Gamma},  f^{*}\nu^{*}i^{*}\omega_Y) = 0$, so the section $t_{\Gamma}$ is unique. 
  As $\operatorname{char} k = 0$ and
  \begin{equation}
  \label{eq:30}
    \tr \circ f^{*} = \deg (f) \cdot \id,
  \end{equation}
  the pullback $f^{*}$ is injective on $H^1$, so $T_D$ is the trivial $\nu^{*}i^{*}\omega_Y$-torsor. By \eqref{eq:31}, it has a unique global section $s_D$. Since $f^{*}T_D$ has a unique global section, $t = f^{*}s_D$.
  Let $j_D\colon \tilde{D} \xrightarrow{s_D} T_D \to X_c$. The normalisation map $\tilde{\Gamma} \to \Gamma \to X_c$ is exactly the map 
  \begin{equation}
  \label{eq:123}
    \tilde{\Gamma} \xrightarrow{f} \tilde{D} \xrightarrow{j_D} X_c.
  \end{equation}
  Since $f$ is surjective, $j_D$ and $j_D \circ f$ have the same image in $X_c$. The image of $\tilde{\Gamma} \to X_c$ is $\Gamma$, so 
  \begin{equation}
  \label{eq:128}
    j_D(\tilde{D}) = \Gamma\subset X_c.
  \end{equation}
  Hence, any curve in $X_c$ with image $D\subset Y$ is equal to $j_D(\tilde{D})$. 
\end{proof}

\begin{lemma}\label{lemma:q* closed immersion on stacks}
  The inclusion of $\mathfrak{M}_{X_c}$ into $\mathfrak{M}_Y$ as the derived critical locus of the function in Theorem \ref{thm:global critical chart} is given by the pushforward along $q\colon X_c \to Y$. That is, the following diagram commutes
  \begin{equation}
  \label{eq:54}
  \begin{tikzcd}
    \mathfrak{M}_{X_c} \arrow[r,"\sim"] \arrow[drr,"q_{*}" swap, bend right = 10] & \mathfrak{M}_{\Pi_3(\mathsf{QCoh} Y, \tilde{c})} \arrow[r,"\sim"]& \Crit(f_c) \arrow[d,hook] \\ & & \mathfrak{M}_Y.
  \end{tikzcd}
  \end{equation}
  In particular, the inclusion $q_{*}$ is a closed immersion onto the critical locus of $f_c$. 
\end{lemma}

\begin{proof}
  Let $A_c\coloneq q_{*}\O_{X_c}$.   The pushforward functor $q_{*}\colon \QCoh (X_c) \to \QCoh(Y)$ forgets the $A_c$-module structure. A sheaf on $X_c$ corresponds to an $A_c$-module on $Y$ and $q_{*}$ sends it to the underlying $\O_Y$-module. By Theorem \ref{thm:km cy completion}, there is an equivalence of $k$-linear stable $\infty$-categories
  \begin{equation}
  \label{eq:33}
    \QCoh(X_c) \simeq \Pi_3(\QCoh(Y),\tilde{c}),
  \end{equation}
  where $\tilde{c}\in \HC^-_1(\QCoh(Y))$ denotes the negative cyclic lift of $c\in \HH_1(\QCoh(Y))$.
  Moreover, in the proof of Theorem \ref{thm:km cy completion}, the equivalence is obtained by identifying the deformed Calabi--Yau completion with 
  \begin{equation}
  \label{eq:57}
    \mathrm{RMod}_{\mathcal{A}_{0,c}}(\QCoh(Y))
  \end{equation}
  and identifying $\mathcal{A}_{0,c}$ with $q_{*}\O_{X_c}$. Hence, this equivalence identifies the forgetful functor 
  \begin{equation}
  \label{eq:58}
    \Pi_3(\QCoh(Y),\tilde{c}) \to \QCoh(Y)
  \end{equation}
  with the pushforward functor $q_{*}$.
  Furthermore, by \cite[Proposition 6.11]{BozecCalaqueScherotzke2024} and \cite[Example 6.15]{BozecCalaqueScherotzke2024}, the induced equivalence on moduli stacks
  \begin{equation}
  \label{eq:47}
    \mathfrak{M}_{\Pi_3(\QCoh(Y),\tilde{c})} \simeq \Crit(f_c)
  \end{equation}
  is over $\mathfrak{M}_Y$, too. 
\end{proof}

\begin{lemma}\label{lemma:q* closed immersion}
  The morphism $\iota \colon M_{X_c}^{\mathrm{ss}}(r,m) \to M_Y^{\mathrm{ss}}(r,m)$ is a closed immersion and 
  \begin{equation}
  \label{eq:201}
  \begin{tikzcd}
    \mathfrak{M}_{X_c}^{\mathrm{ss}}(r,m) \arrow[d]\arrow[r,"q_{*}"] & \mathfrak{M}_Y^{\mathrm{ss}}(r,m) \arrow[d]\\
    M_{X_c}^{\mathrm{ss}}(r,m) \arrow[r,"\iota"] & M_Y^{\mathrm{ss}}(r,m)
  \end{tikzcd}  
  \end{equation}
  commutes.
\end{lemma}

\begin{proof}
  By Lemma \ref{lemma:q* closed immersion on stacks} and Proposition \ref{prop:semistable}, $q_{*}$ is a closed immersion. Let $Z\subset M_Y^{\mathrm{ss}}(r,m)$ be its image. By \cite[Lemma 4.14]{Alper2013},  $\mathfrak{M}_{X_c}^{\mathrm{ss}}(r,m) \to Z$ is a good moduli space. By Remark \ref{rem:factor-through-good-moduli},
  \begin{equation}
  \label{eq:205}
    Z \cong M_{X_c}^{\mathrm{ss}}(r,m).
  \end{equation}
  Therefore, $\iota$  is a closed immersion.
\end{proof}

\begin{lemma}\label{lemma:b separated}
  The morphism $b\colon S_c \to B_Y$  is separated over $k$.
\end{lemma}

\begin{proof}
  We show that $S_c \subset \Chow_{r[C]}(X_c)$ is separated over $k$. Choose a projective compactification $\overline{X}_c$ of $X_c$. The Chow scheme  $\Chow_{r[C]}(\overline{X}_c)$ is separated over $k$. As 
  \begin{equation}
  \label{eq:198}
    \Chow_{r[C]}(X_c) \subset \Chow_{r[C]}(\overline{X}_c)
  \end{equation}
  is a locally closed subscheme, it, too is separated over $k$. Thus, the subscheme $S_c \subset \Chow_{r[C]}(X_c)$ is separated over $k$, and therefore so is $b$.
\end{proof}

\begin{lemma}\label{lemma:HC proper}
  The Hilbert--Chow morphism
  \begin{equation}
  \label{eq:90}
    \HC\colon  M_{X_c}^{\mathrm{ss}}(r,m) \to S_c
  \end{equation}
  is proper over $k$.
\end{lemma}

\begin{proof}
  Consider the commutative diagram 
  \begin{equation}
  \label{eq:101}
  \begin{tikzcd}[column sep=tiny]
    M_{X_c}^{\mathrm{ss}}(r,m) \arrow[rr,"\HC"] \arrow[dr,"h_{X_c}" swap] && \quad S_c \quad \arrow[dl, "b"] \\ & B_Y. &
  \end{tikzcd}
  \end{equation}
  Because the Hitchin morphism  $h_Y$ is proper and by Lemma \ref{lemma:q* closed immersion}, $\iota$ is proper, so is $h_{X_c} = h_Y \circ \iota$. By Lemma \ref{lemma:b separated}, $b$ is separated. By \cite[Tag 01W6]{stacks}, this implies that $\HC\colon M_{X_c}^{\mathrm{ss}}(r,m) \to S_c$ is proper.
\end{proof}

\begin{lemma}\label{lemma:Sc finite type}
  The scheme $S_c$  is of finite type over $k$. 
\end{lemma}

\begin{proof}
  As,  $M^{\mathrm{ss}}_Y(r,m)$ is of finite type over $k$, by Lemma \ref{lemma:q* closed immersion} so is $M^{\mathrm{ss}}_{X_c}(r,m)$. Hence, it is quasi-compact.
  By Lemma \ref{lemma:HC proper}, the Hilbert--Chow morphism 
  \begin{equation}
  \label{eq:49}
    \HC\colon M^{\mathrm{ss}}_{X_c}(r,m)\to S_c
  \end{equation}
  is proper. Since $S_c$  is the scheme-theoretic image of $\HC$, this morphism is surjective on underlying topological spaces.   Thus, $S_c$ is quasi-compact. Since $S_c$ is a closed subscheme of $\Chow_{r[C]}(X_c)$, which is locally of finite type over $k$, the scheme $S_c$ is locally of finite type over $k$. Being also quasi-compact, it is of finite type over $k$.
\end{proof}

\begin{proposition}\label{prop:b finite}
  The morphism $b \coloneq \rho_Y^{-1} \circ q_{\#}|_{S_c}\colon S_c \to B_Y$ is finite.
\end{proposition}

\begin{proof}
  We prove that $b$ is proper and quasi-finite.
  
First we prove properness. By Lemma \ref{lemma:Sc finite type}, $S_c$ is of finite type over $k$.  Since $B_Y$ is a finite-dimensional affine space over $k$, every $k$-morphism $S_c \to B_Y$ is of finite type. In particular, $b$ is of finite type. By Lemma \ref{lemma:b separated}, $b$ is separated. Moreover, $h_{X_c} \coloneq h_Y \circ \iota$ is universally closed, since $\iota $ is a closed immersion and $h_Y$ is proper. Since $\HC$ is surjective onto $S_c$, \cite[Tag 03GN]{stacks} implies that $b$ is universally closed. Since $b$ is separated and of finite type, $b$ is proper.

Now we show that $b$ is quasi-finite. As $b$ is of finite type, it is enough to show that $b$ has finite geometric fibres. Let $\bar{a}\colon \operatorname{Spec}K\to B_Y$ be a geometric point, with $K$ an algebraically closed field. After base change to $K$, write the corresponding spectral cycle as
\begin{equation}\label{eq:spectral-cycle-fiber}
        C_{\bar{a}}=\sum_i m_iD_i
      \end{equation}
with $D_i\subset Y_K$ integral proper curves. By Proposition \ref{prop:unique lift}, for each $D_i$ there is at most one integral curve $\Gamma_i\subset X_{c,K}$ mapping to $D_i$.  Since no proper curve in $X_{c,K}$ is contracted by $q_K$, every cycle $\Gamma$ satisfying
\begin{equation}
\label{eq:209}
  q_{K,\#}\Gamma=C_{\bar{a}}
\end{equation}
is forced to be
\begin{equation}
\label{eq:210}
  \Gamma=\sum_i m_i\Gamma_i,
\end{equation}
if the lifts $\Gamma_i$ exist. Hence every geometric fibre of $b$ has at most one point. Since each geometric fibre is of finite type over $K$, it is $0$-dimensional and therefore finite. Thus, $b$ is quasi-finite. Since $b$ is proper and quasi-finite, it is finite.
\end{proof}

\subsection{BPS-Decomposition}\label{sec:BPS decomposition}

From now on, let $k=\C$.

We shall use the BPS decomposition theorem in the following form.
For a fixed slope $\mu\in\Q$, set 
\begin{equation}
\label{eq:106}
  \mathfrak{M}_Y^{\mathrm{ss}}(\mu) \coloneq \coprod_{m/r = \mu} \mathfrak{M}_Y^{\mathrm{ss}}(r,m), \qquad M_Y^{\mathrm{ss}}(\mu) \coloneq \coprod_{m/r = \mu} M_Y^{\mathrm{ss}}(r,m).
\end{equation}
and let $p\colon \Mf_Y^{\mathrm{ss}}(\mu)\to M_Y^{\mathrm{ss}}(\mu)$ be the good moduli space morphism. 
The direct sum of semistable Higgs bundles of fixed slope defines finite morphisms 
\begin{equation}
\label{eq:107}
  \oplus\colon (M_Y^{\mathrm{ss}}(\mu))^n \to M_Y^{\mathrm{ss}}(\mu),
\end{equation}
and hence symmetric-convolution functors $\Sym_{\boxtimes_{\oplus}}$ on the locally finite bounded-below derived category $D^{\geq, \mathrm{lf}}(\MMHM(M_Y^{\mathrm{ss}}(\mu)))$.

If a function $F_{\mu} \colon M_Y^{\mathrm{ss}}(\mu) \to \A^1$ satisfies additivity under direct sums
\begin{equation}
\label{eq:108}
  F_{\mu}(A\oplus B) = F_{\mu}(A) + F_{\mu}(B),
\end{equation}
then Thom--Sebastiani identifies the monodromic vanishing  cycles of $F_{\mu}$ with the symmetric-convolution operation.
This is the mechanism behind the BPS decomposition theorem quoted in Section \ref{sec:BPS decomposition}. We later take $F_\mu$ to be defined on each component by
\begin{equation}
\label{eq:224}
  F_\mu(E,\theta)\coloneq \sum_{k\geq1}\frac{1}{k}c_k\left(\tr(\theta^k)\right).
\end{equation}
It is additive under direct sums, and on $M_Y^{\mathrm{ss}}(r,m)$ it equals $\tilde{f}_c\circ h_Y$.

For $\mathfrak{M}_Y^{\mathrm{ss}}(\mu)$ write
\begin{equation}
\label{eq:112}
  \IC_{\mathfrak{M}_Y^{\mathrm{ss}}(\mu)} \coloneq \bigoplus_{m/r=\mu} \IC_{\mathfrak{M}_{Y}^{\mathrm{ss}}(r,m)}.
\end{equation}

\begin{theorem}[{\cite[Theorem 4.6]{KinjoKoseki2026}}]\label{thm:BPS decomposition}
  Let  $F_{\mu} \colon M_Y^{\mathrm{ss}}(\mu) \to \A^1$ be a regular function satisfying 
  \begin{equation}
  \label{eq:100}
     \forall A,B \in M_Y^{\mathrm{ss}}(\mu): \quad    F_{\mu}(A\oplus B) = F_{\mu}(A) + F_{\mu}(B) .
  \end{equation}
  Then there is an isomorphism  
  \begin{equation}
  \label{eq:7}
    \mathcal{H}\left(p_{*}\phi_{F_{\mu} \circ p}^{\mathrm{mmhm}} \IC_{\mathfrak{M}_Y^{\mathrm{ss}}(\mu)}\right) \cong \Sym_{\boxtimes_{\oplus}}\left( H^{*}(B\C^{\times})_{\vir} \otimes \phi_{F_{\mu}}^{\mathrm{mmhm}}  \IC_{M_Y^{\mathrm{ss}}(\mu)}\right)
  \end{equation}
  in $D^{\geq, \mathrm{lf}}(\mathrm{MMHM}(M^{\mathrm{ss}}_Y(\mu)))$.
\end{theorem}

\begin{proposition}\label{prop:vanishing cycle dimensional reduction}
  There is an isomorphism 
  \begin{equation}
  \label{eq:22}
    \iota_{*} \left(\phi_{M_{X_c}^{\mathrm{ss}}(r,m)}^{\mathrm{mmhm}}\right) \simeq \phi_{\tilde{f}_c\circ h_Y}^{\mathrm{mmhm}} \left( \IC_{M_Y^{\mathrm{ss}}(r,m)}\right) \in \MMHM\left(M_Y^{\mathrm{ss}}(r,m)\right).
  \end{equation}
  After rationalisation, this gives the corresponding isomorphism of perverse sheaves.
\end{proposition}

\begin{proof}
  To avoid confusion, we write 
  \begin{equation}
  \label{eq:223}
    p_X\colon \mathfrak{M}_X^{\mathrm{ss}}(r,m) \to M_X^{\mathrm{ss}}(r,m) \qquad \text{and} \qquad p_Y\colon \mathfrak{M}_Y^{\mathrm{ss}}(r,m) \to M_Y^{\mathrm{ss}}(r,m).
  \end{equation}
  Let
  \begin{equation}
  \label{eq:76}
    j\coloneq q_*\colon    \mathfrak{M}_{X_c}^{\mathrm{ss}}(r,m)    \hookrightarrow    \mathfrak{M}_Y^{\mathrm{ss}}(r,m)
  \end{equation}
  be the closed immersion of Proposition \ref{prop:semistable}.
  We show that both sides of \eqref{eq:22} are isomorphic to 
  \begin{equation}
  \label{eq:35}
    \mathcal{H}^1\left( p_{Y,*} \left( \phi_{\tilde{f}_c\circ \tilde{h}_Y}^{\mathrm{mmhm}} \left( \IC_{\mathfrak{M}_Y^{\mathrm{ss}}(r,m)}  \right)  \right)  \right),
  \end{equation} 
  Taking first cohomology in \eqref{eq:7} in Theorem \ref{thm:BPS decomposition} with $F_{\mu} \coloneq \tilde{f}_c \circ h_Y$ and restricting to $M_Y^{\mathrm{ss}}(r,m)\subset M_Y^{\mathrm{ss}}(\mu)$ gives
  \begin{equation}
  \label{eq:221}
    \mathcal{H}^1\left(
      p_{Y,*}
      \phi_{\tilde f_c\circ\tilde h_Y}^{\mathrm{mmhm}}
      \left(\IC_{\mathfrak{M}_Y^{\mathrm{ss}}(r,m)}\right)
    \right)
    \simeq
    \phi_{\tilde f_c\circ h_Y}^{\mathrm{mmhm}}
    \left(\IC_{M_Y^{\mathrm{ss}}(r,m)}\right).
  \end{equation}
  By Propositions \ref{prop:semistable} and  \ref{prop:naturally oriented}, together with   \cite[Proposition 2.4 and Remark 2.5]{KinjoKoseki2026}, 
  \begin{equation}
  \label{eq:46}
    j_*\phi_{\mathfrak{M}_{X_c}^{\mathrm{ss}}(r,m)}^{\mathrm{mmhm}}  \simeq  \phi_{\tilde{f}_c\circ\tilde{h}_Y}^{\mathrm{mmhm}} \left(\IC_{\mathfrak{M}_Y^{\mathrm{ss}}(r,m)}\right).
  \end{equation}
  Note that $p_Y \circ j = \iota \circ p_X$ and since $\iota$ is a closed immersion by
  Lemma \ref{lemma:q* closed immersion}, $\iota_*$ is exact. Hence
  \begin{align}
  \label{eq:222}
    \iota_*\left(\phi_{M_{X_c}^{\mathrm{ss}}(r,m)}^{\mathrm{mmhm}}\right)  &\simeq  \mathcal{H}^1\left(  p_{Y,*}j_*  \phi_{\mathfrak{M}_{X_c}^{\mathrm{ss}}(r,m)}^{\mathrm{mmhm}} \right) \\
    &\simeq \mathcal{H}^1\left(p_{Y,*} \phi_{\tilde{f}_c\circ\tilde{h}_Y}^{\mathrm{mmhm}} \left(\IC_{\mathfrak{M}_Y^{\mathrm{ss}}(r,m)}\right) \right).
  \end{align}
  Combining \eqref{eq:221} and \eqref{eq:222} proves \eqref{eq:22}.
  Rationalisation gives the corresponding isomorphism of perverse sheaves.
\end{proof}

\subsection{Proof of $\chi$-Independence}

\begin{theorem}\label{thm:chi-independence on Y}
  Let $r\in\mathbb{Z}_{>0}$, $m, m'  \in \Z$.
  Let $C$ be a smooth projective curve over $\C$, $L$ a line bundle on $C$ of degree $> 2g(C)-2$, and $\pi\colon Y \coloneq \tot_C(L)\to C$ the total space of $L$. Fix a class $c \in H^1(Y,\omega_Y)$ and denote by $q\colon X_c \to Y$ the total space of the $\omega_Y$-torsor associated to it.  For $n \in \Z$, set
  \begin{equation}
  \label{eq:38}
    h_{X_c}\coloneq h_Y\circ \iota \colon M_{X_c}^{\mathrm{ss}}(r,n) \to B_Y.
  \end{equation}
  Then 
  \begin{equation}
  \label{eq:25}
    h_{Y,*}^c(\phi_{M_{X_c}^{\mathrm{ss}}(r,m)}) \cong h_{Y,*}^c(\phi_{M_{X_c}^{\mathrm{ss}}(r,m' )}).
  \end{equation}
\end{theorem}

\begin{proof}
  By Proposition \ref{prop:semistable}, $\mathfrak{M}_{X_c}^{\mathrm{ss}}(r,m)$ and $\mathfrak{M}_{X_c}^{\mathrm{ss}}(r,m' )$ can be written as global critical loci inside $\mathfrak{M}_Y^{\mathrm{ss}}(r,m)$ and $\mathfrak{M}_Y^{\mathrm{ss}}(r,m' )$. By Remark \ref{rem:factor-through-hitchin}, there is a map $\tilde{f}_c \colon B_Y \to \A^1$ such that $f_c  = \tilde{f}_c \circ \tilde{h}_Y$. Thus, by Proposition \ref{prop:vanishing cycle dimensional reduction}, 
  \begin{equation}
  \label{eq:26}
    h_{Y,*}^c  \left(\phi_{M_{X_c}^{\mathrm{ss}}(r,m)}\right) \simeq h_{Y,*}\left( \phi_{\tilde{f}_c \circ h_Y}  \left( \IC_{M_Y^{\mathrm{ss}}(r,m)} \right)\right). 
  \end{equation}
  Taking vanishing cycles commutes with proper pushforwards by Proposition \ref{prop:proper-pushforward-vanishing-cycles}, so
  \begin{equation}
  \label{eq:75}
     h_{Y,*}\left( \phi_{\tilde{f}_c \circ h_Y} \left(\IC_{M_Y^{\mathrm{ss}}(r,m)}\right)\right)\simeq \phi_{\tilde{f}_c} \left( h_{Y,*} \left( \IC_{M_Y^{\mathrm{ss}}(r,m)} \right)\right).
  \end{equation}
  As $\deg L > 2g(C) -2$, we can apply Theorem \ref{thm:ms}:
  \begin{equation}
  \label{eq:73}
    h_{Y,*}  \IC_{M_Y^{\mathrm{ss}}(r,m)}  \cong h_{Y,*}  \IC_{M_Y^{\mathrm{ss}}(r,m' )} .
  \end{equation}
  Hence, applying this to \eqref{eq:26} and \eqref{eq:75},
  \begin{equation}
  \label{eq:23}
    h_{Y,*}^c \phi_{M_{X_c}^{\mathrm{ss}}(r,m)} \cong  h_{Y,*}^c \phi_{M_{X_c}^{\mathrm{ss}}(r,m' )}. 
  \end{equation}
\end{proof}

\begin{corollary}[= Theorem \ref{thm:B}]\label{cor:chi-independence}
  In the setting above, for every $i \in \Z$,
  \begin{equation}
  \label{eq:1}
    \chi(\mathcal{H}^i(\HC_{*}(\phi_{M_{X_c}^{\mathrm{ss}}(r,m)}))) = \chi(\mathcal{H}^i(\HC_{*}(\phi_{M_{X_c}^{\mathrm{ss}}(r,m' )})))
  \end{equation}
\end{corollary}

\begin{proof}
  By Proposition \ref{prop:b finite}, $b$ is finite. Thus,
  \begin{align}
  \label{eq:36}
    \chi(\mathcal{H}^i(\HC_{*}(\phi_{M_{X_c}^{\mathrm{ss}}(r,m)}))) &= \chi(\mathcal{H}^i(b_{*}\HC_{*}(\phi_{M_{X_c}^{\mathrm{ss}}(r,m )}))) \\ 
    \label{eq:37}
    &=\chi(\mathcal{H}^i(h_{Y,*}^{c}(\phi_{M_{X_c}^{\mathrm{ss}}(r,m )}))).
  \end{align}
  Theorem \ref{thm:chi-independence on Y} concludes the proof. 
\end{proof}

\printbibliography


\end{document}